\documentclass[12pt]{article}
\usepackage{graphicx,verbatim,array,multicol,courier}
\usepackage{amsmath,amssymb,amsthm,mathrsfs}
\usepackage[pdftex,bookmarks,colorlinks]{hyperref}
\usepackage{fullpage}
\usepackage{topcapt}
\usepackage{chicago}
\usepackage{color}
\usepackage[section]{placeins} 
\usepackage{mathdots}
\usepackage{mathtools} 
\usepackage{epstopdf} 
\usepackage[labelformat=simple]{subfig}
\usepackage{soul} 
\usepackage{booktabs} 
\usepackage{dsfont}
\usepackage{paralist}
\usepackage[dvipsnames]{xcolor}
\usepackage{enumitem}

\definecolor{navy}{rgb}{0,0,0.502}

\hypersetup{colorlinks,%
citecolor=blue,%
filecolor=green,%
linkcolor=red,%
urlcolor=magenta,%
}

\newtheorem{lemma}{Lemma}[section]

\newtheorem{theorem}[lemma]{Theorem}

\newtheorem{prop}[lemma]{Proposition}
\newtheorem{exam}[lemma]{\normalfont \scshape
 Example}
\newtheorem{rem}[lemma]{\normalfont \scshape Remark}

\newenvironment{example}{\begin{exam}}{\end{exam}}

\DeclarePairedDelimiter\floor{\lfloor}{\rfloor}

\newcommand{\R}{\mathbb{R}}
\newcommand{\N}{\mathbb{N}}

\newcommand{\D}{\mathbb{D}}

\newcommand{\abs}[1]{\left\vert#1\right\vert}
\newcommand{\set}[1]{\left\{#1\right\}}

\newcommand{\bfx}{\boldsymbol{x}}
\newcommand{\bfzero}{\boldsymbol{0}}
\newcommand{\bfinfty}{\boldsymbol{\infty}}

\newcommand{\bfone}{\boldsymbol{1}}
\newcommand{\bfa}{\boldsymbol{a}}
\newcommand{\bfb}{\boldsymbol{b}}

\newcommand{\bfA}{\boldsymbol{A}}
\newcommand{\bfD}{\boldsymbol{D}}

\newcommand{\bfM}{\boldsymbol{M}}
\newcommand{\bfI}{\boldsymbol{I}}

\newcommand{\bfU}{\boldsymbol{U}}
\newcommand{\bfu}{\boldsymbol{u}}
\newcommand{\bfB}{\boldsymbol{B}}
\newcommand{\bfV}{\boldsymbol{V}}

\newcommand{\bfX}{{\boldsymbol{X}}}

\newcommand{\bfZ}{\boldsymbol{Z}}
\newcommand{\bfz}{\boldsymbol{z}}

\newcommand{\bfLambda}{\boldsymbol{\Lambda}}

\newcommand{\bfSigma}{\boldsymbol{\Sigma}}

\newcommand{\bfOmega}{\boldsymbol{\Omega}}
\newcommand{\bfGamma}{\boldsymbol{\Gamma}}
\newcommand{\bfDelta}{\boldsymbol{\Delta}}
\newcommand{\bfmu}{\boldsymbol{\mu}}

\newcommand{\bfxi}{\boldsymbol{\xi}}
\newcommand{\bfvarrho}{\boldsymbol{\varrho}}

\newcommand{\bfsigma}{\boldsymbol{\sigma}}

\newcommand{\calI}{\mathcal I}
\newcommand{\bcalI}{{\mathcal I}^\complement}
\newcommand{\calJ}{\mathcal{J}}

\newcommand{\calD}{\mathcal D}

\DeclareMathOperator{\cov}{Cov}
\DeclareMathOperator{\cor}{Cor}

\DeclareMathOperator{\expect}{E}

\DeclareMathOperator{\Pro}{Pr}

\def\indic{\mathds{1}}
\def\bzero{{\bf 0}}
\def\diff{\mbox{d}}

\newcommand{\isrec}{R}

\newcommand{\iscrec}{R^{\text{CR}}}

\def\mins{{\scriptscriptstyle{\setminus}}}

\allowdisplaybreaks[4]

\begin{document}

\title{Records for Some Stationary Dependent Sequences}%
\author{M. Falk, A. Khorrami and S. A. Padoan}

\maketitle

\begin{abstract}
For a zero-mean, unit-variance second-order stationary univariate Gaussian process we derive the probability that a record at the time $n$, say $X_n$, takes place and derive its distribution function.
We study the joint distribution of the arrival time process of records and the distribution of the increments between the first and second record, and the third and second record and we compute the expected number of records. We also consider two consecutive and non-consecutive records, one at time $j$ and
one at time $n$ and we derive the probability that the
joint records $(X_j,X_n)$ occur as well as their distribution function.
The probability that the records $X_n$ and $(X_j,X_n)$ take place and
the arrival time of the $n$-th record, are independent of the marginal distribution function, provided that it is continuous. These results actually hold for a second-order stationary process with Gaussian copulas. We extend some of these results to the case of a multivariate Gaussian process.
Finally, for a strictly stationary process satisfying some mild conditions on
the tail behavior of the common marginal distribution function $F$ and the long-range dependence of the extremes of the process, we derive the asymptotic probability that the record $X_n$ occurs and derive its distribution function.

\end{abstract}

\section{Introduction}

Let $\{X_n, n\geq 1\}$ be a sequence of identically distributed random variables (rvs), and denote by
$F$ the common univariate marginal distribution function.
For any $i,j\in\N$,
set $M_{i:j}:=\max(X_i,\ldots,X_j)$.
For simplicity, we set $M_j :=M_{1:j}$, that is $M_j:=\max(X_1,\ldots,X_j)$.
The rv $X_n$ is a record if $X_n>M_{n-1}$. Such an event is coded by the indicator function
$\isrec_n:=\indic(X_n \text{ is a record})$.
When $X_1,X_2,\ldots$ are independent,
many results on records are already known (e.g., \citeNP{gal87}; \citeNP{arnbn98}; \citeNP[Ch. 4]{resn08}; \citeNP{barakat2017}; \shortciteNP{falkkp2018}). In the multivariate case various definitions of records are possible and have been investigated both in the past and more recently, see e.g., \citeN{golres89}, \citeN{hashhue05}, \citeN{hwang2010}, \shortciteN{domfalzot18} to name a few. In this work we consider \emph{complete records}; these are random vectors which are univariate records in each component.
Precisely, let $\{\bfX_n,n\geq 1\}$ be a strictly stationary sequence of $d$-dimensional random vectors (rvs) $\bfX_n=(X_n^{(1)},\ldots,X_n^{(d)})\in\R^d$. Let $F$ be the common joint distribution function
of $\bfX_n$ with margins $F_i$, $1\leq i\leq d$. The rv $\bfX_n$ is a complete record if
$$
\bfX_n> \max_{1\leq i\leq n-1} \bfX_i,
$$
where the maximum is computed componentwise.
We denote the rv coding the occurrence of a complete record at time $n$ by
$\iscrec_n:=\indic(\bfX_n \text{ is a complete record})$.

Except for \citeN{haiman1987}, \shortciteN{haiman1998}, as far as we know, most of the available results on records concern sequences of independent random variables or vectors.
In the present work we derive
some new results on the records of a stationary sequence of dependent random variables and dependent random vectors, under appropriate conditions of the dependence structure.

At first we consider a univariate second-order stationary 
Gaussian process with zero-mean, unit-variance.  
This means that for every $n=1,2,\ldots$, $\expect(X_n)=0$, $\expect^2(X_n)=1$ and the autocovariance of the process is translation-invariant depending only on the time difference, i.e.
for every $i,j$, $\rho_{i,j}=\expect(X_iX_j)=\expect(X_0X_{j-i})=\rho_{0,j-i}\equiv \rho_{j-i}$, where
$\rho_{j-i}$ is a function only of the separation $j-i$ and for every $m$,
$\rho_{i+m,j+m}=\rho_{j-i}$.
We derive the probability that a record at time $n$, say $X_n$, takes place, and the distribution of $X_ n$, being a record. Furthermore, we derive the joint distribution of the arrival time process of records and more specifically the distribution of the increments between the first and second record and the third and second record.
We compute the expected number of records which, depending on the type of correlation structure of the Gaussian process, can be finite or infinite. We also focus on joint records and we derive  the probability that two consecutive and non-consecutive records at the time $j$ and $n$, say $X_j$ and $X_n$, take place, as well as the joint distribution of $(X_j, X_ n)$, considering they are both records.

We highlight that many of our findings, such as the probability that the records $X_n$ and $(X_j,X_n)$ take place and
the arrival time of the $n$-th record, are independent of the marginal distribution function $F$, provided that is is continuous. As a consequence, the results actually hold for second-order stationary sequences with \emph{Gaussian copulas}. On the contrary the  distribution of a record (two records), conditional to the assumption that it is a record (they are records), however does depend on $F$.

Next we consider a strictly stationary process satisfying some mild conditions on
the tail behavior of the common marginal distribution function $F$ and the long-range dependence of the extremes of the process. More specifically, it is assumed that $F$ is attracted by the so-called Generalized Extreme-Value family of distributions, and that maxima on separated enough intervals within the time span $n$ are approximately independent. Within this setting we derive the probability that $X_n$ is a record, the distribution of $X_ n$ (being a record), and the expected number of records.

We complete the work by considering a zero-mean, unit-variance multivariate second-order stationary Gaussian process. We derive the probability that a complete record at time $n$ occurs, and we compute the distribution of $\bfX_ n$ (being a record), as well as the probability that
two complete records at the time $j$ and $n$ occur, and the joint distribution of
$(\bfX_j,\bfX_ n)$ (being records).

The paper is organized as follows. In Section \ref{sec:prelim} we introduce some notation used throughout the paper and we briefly review some basic concepts on the multivariate closed skew-normal distribution.
In Section \ref{sec:uni_gauss_seq} we present our main results on records for an univariate second-order stationary Gaussian process. In Section \ref{sec:asymptotic} we provide the asymptotic probability
and distribution function of a record at time $n$ for a strictly stationary process that satisfies some
appropriate conditions. Finally, in Section \ref{sec:multivariate} we extend some of the results
derived in Section \ref{sec:asymptotic} to the case of multivariate second-order stationary Gaussian processes.

\section{Univariate Case}
\subsection{Preliminary results and notation}\label{sec:prelim}

Throughout the paper we use the following notation. The symbol $\bfX\sim N_n(\bfmu,\bfSigma)$, $n\in\N$, means an $n$-dimensional random vector that follows a multivariate Gaussian distribution with mean $\bfmu\in\R^n$ and positive-definite covariance matrix $\bfSigma=\bfsigma\bar{\bfSigma}\bfsigma\in\R^{n,n}, \bfsigma:=\text{diag}(\sigma_{11},\dots,\sigma_{nn})$, and
  $\bar{\bfSigma}$ is the correlation matrix. Its cumulative distribution function (cdf) and probability density function
(pdf) are denoted by $\Phi_n(\bfx;\bfmu,\bfSigma)$ and $\phi_n(\bfx;\bfmu,\bfSigma)$ with
$\bfx\in\R^n$. When $\bfmu=\bzero=(0,\ldots,0)^\top$ and $\bfSigma=\bfI$, where $\bfI$ is the identity matrix, we write $\Phi_n(\bfx)$ for simplicity.

We indicate with $\mathbf{1}_{a,b}$  ($\bfzero_{a,b}$) a matrix of dimension $a\times b$ whose elements are all equal to one (zero). We omit the subscripts when the  dimensions of the matrices are clear from the context.
%

We introduce the notion of a multivariate {\it closed skew-normal} (CSN) random vector 
and we do so by using the so-called conditioning representation (\citeNP[Ch. 2]{genton2004}).
Let $\bfU\sim N_m(\bfxi,\bfOmega)$ being independent of $\bfV\sim N_n(\bzero,\bfSigma)$, 
where $\bfxi\in\R^m$, $\bfOmega\in\R^m\times \R^m$ and $\bfSigma\in \R^n\times \R^n$.
Let $\bfDelta\in\R^n\times\R^m$, then
$$
\begin{pmatrix}
\bfU\\
\bfDelta\bfU+\bfV
\end{pmatrix}
\sim
N_{m+n}
\left(
\begin{pmatrix}
\bfxi\\
\bzero
\end{pmatrix}
,
\begin{pmatrix}
\bfOmega & \bfOmega\bfDelta^\top\\
\bfDelta\bfOmega & \bfGamma
\end{pmatrix}
\right),
$$
where $\bfGamma = \bfSigma + \bfDelta\bfOmega\bfDelta^\top$.
Define $\bfX$ equal to $\bfU$, under the condition that $\bfDelta\bfU+\bfV > \bfmu$, denoted by $\bfX=(\bfU|\bfDelta\bfU+\bfV > \bfmu)$, where $\bfmu\in\R^n$. The $m$-dimensional random vector $\bfX$ follows a multivariate closed skew-normal distribution, in symbols $\bfX\sim CSN_{m,n}(\bfxi,\bfOmega,\bfDelta,\bfmu,\bfSigma)$, whose pdf is, for all $\bfx\in\R^m$,
\begin{equation}\label{eq: CSN density}
\psi_{m,n}(\bfx;\bfxi,\bfOmega,\bfDelta,\bfmu,\bfSigma)=\frac{\phi_m(\bfx-\bfxi;\bfOmega)
\Phi_n(\bfDelta(\bfx-\bfxi);\bfmu,\bfSigma)}{\Phi_n(\bzero;\bfmu,\bfGamma)}.
\end{equation}
We denote the cdf of $\bfX$ by $\Psi_{m,n}(\bfx;\bfxi,\bfOmega,\bfDelta,\bfmu,\bfSigma)$.
When $\bfxi=\bzero$, $\bfOmega=\bfI$ and $\bfmu=\bzero$, we omit them among the parameters for simplicity and we write
$\Psi_{m,n}(\bfx;\bfDelta,\bfSigma)$ and $\psi_{m,n}(\bfx;\bfDelta,\bfSigma)$ instead.
%
We recall that the closed skew-normal distribution is also known in the literature  as the unified multivariate skew-normal distribution, which simply uses a different parametrization (e.g, Ch. 7.1.2 in \citeNP{azzalini2013skew}). The exposition of our results benefits from the parametrization used by the closed skew-normal distribution.

We recall that if $\bfX\sim CSN_{m,n}(\bfxi,\bfOmega,\bfDelta,\bfmu,\bfSigma)$ then
\begin{equation}\label{eq: CSN cdf}
\Psi_{m,n}(\bfx;\bfxi, \bfOmega, \bfDelta,\bfSigma)=\frac{\Phi_{n+m}(\tilde{\bfx};\tilde{\bfOmega})}{\Phi_n(\bfzero;\bfmu,\bfGamma)},
\end{equation}
where
\[
\tilde{\bfx}=
\begin{pmatrix}
-\bfmu\\
\bfx-\bfxi
\end{pmatrix},
\quad
\tilde{\bfOmega}=
\begin{pmatrix}
\bfGamma & -\bfOmega\bfDelta^\top\\
\bfDelta\bfOmega & \bfOmega
\end{pmatrix},
\]
see \citeN{azzalini2010}. Furthermore, for $\bfb\in \R^m$
and $\bfA\in\R^{q,m}$ then,
\begin{eqnarray}
\bfb+\bfX &\sim& CSN_{m,n}(\bfxi+\bfb,\bfOmega,\bfDelta,\bfmu,\bfSigma)\\
\bfA\bfX &\sim& CSN_{q,n}(\bfA\bfxi,\bfOmega^{*},\bfDelta^*,\bfmu,\bfSigma^*)\label{eq: affine CSN}
\end{eqnarray}
where $\bfOmega^{*}=\bfA\bfOmega \bfA^\top$, $\bfDelta^{*}=\bfDelta\bfOmega \bfA^\top {\bfOmega^{*}}^{-1}$ and
$\bfSigma^{*}=\bfGamma - \bfDelta^{*}\bfA\bfOmega\bfDelta^\top$,
(see Ch. 2 in \citeNP{genton2004} for details).

\subsection{Records of dependent univariate Gaussian sequences}\label{sec:uni_gauss_seq}

Let $\{X_n, n\geq 1\}$ be a second-order stationary Gaussian sequence of dependent rvs.
Without loss of generality, assume for simplicity that
$\expect(X_i)=0$, $\expect(X^2_i)=1$ for every $1\leq i\leq n$.
Throughout the paper we will refer to such a process as a stationary
standard Gaussian (SSG) sequence.
For any $n\in\N$, let $\calI\subset\{1,\dots,n\}$ and
$\calI^\complement=\{1,\dots,n\}\setminus{\calI}$
identify the $|\calI|$-dimensional and $|\calI^\complement|$-dimensional subvector partition such
that $\bfX=(X_1,\ldots,X_n)^{\top}=(\bfX_{\calI}^{\top},\bfX_{\calI^\complement}^{\top})^{\top}$,
with corresponding partition of the parameter $\bar{\bfSigma}$. By $|A|$ we denote the number of elements of a set $A$.

Our results rely on the following well-known important result on the conditional distribution derived from
joint Gaussian distribution. Precisely, let $\bfX=(\bfX_{\calI}^{\top},\bfX_{\calI^\complement}^{\top})^{\top}\sim N_n(\bfmu,\bfSigma)$ with corresponding partition of the parameters $\bfmu$ and
$\bfSigma$, then in \citeN[Theorem 2.5.1]{ander84} it is established that
the conditional distribution of $\bfX_{\calI^\complement}$ given that $\bfX_{\calI}=\bfx_{\calI}$, is
for all $\bfx_{\calI}\in\R^{|{\calI}|}$,
\begin{equation}\label{eq:cond_gauss}
\begin{split}
\bfX_{\calI^\complement}\vert \bfX_{\calI}&=\bfx_{\calI}\sim N_{|\calI^\complement|}
\left(\bfmu_{\calI^\complement},\bfSigma_{\calI^\complement,\calI^\complement;\calI}\right),\\
\bfmu_{\calI^\complement}&=\bfSigma_{\calI^\complement,\calI}\bar{\bfSigma}_{\calI,\calI}^{-1}\bfx_{\calI},\\
\bfSigma_{\calI^\complement,\calI^\complement;\calI}&=
\bar{\bfSigma}_{\calI^\complement,\calI^\complement}-
\bfSigma_{\calI^\complement,\calI}\bar{\bfSigma}_{\calI,\calI}^{-1}\bfSigma_{\calI,\calI^\complement}.
\end{split}
\end{equation}
Furthermore, we denote the related correlation matrix by
$$
\bar{\bfSigma}_{\calI^\complement,\calI^\complement;\calI}=
\bfsigma_{\calI^\complement,\calI^\complement;\calI}^{-1}\bfSigma_{\calI^\complement,\calI^\complement;\calI}
\bfsigma_{\calI^\complement,\calI^\complement;\calI}^{-1},
$$
where $\bfsigma_{\calI^\complement,\calI^\complement;\calI}=\text{diag}(\bfSigma_{\calI^\complement,\calI^\complement;\calI})$.
For any $j\in \{a,\ldots,b\}$, when $\calI=\{j\}$  we simplify the notation writing $X_j$ and
$\bfX_{a:b\mins j}=(X_a,\ldots,X_{j-1},X_{j+1},\ldots,X_b)^{\top}$. When $j=a$ or $j=b$ we further
simplify the notation by $\bfX_{2:b}=(X_2,\ldots X_b)^{\top}$ and
$\bfX_{1:b-1}=(X_1,\ldots,X_{b-1})^{\top}$.

In our first result we compute the probability that $X_n$ is a record together with its distribution. It is well known that $\Pro(\isrec_n=1)=1/n$
in the case of independent rv with identical continuous df  (see e.g., \citeNP{gal87})
and that the distribution of $X_n$, given that it is a record, equals that of the largest observation among $X_1,\dots,X_n$ \shortcite{falkkp2018}.

\begin{prop}\label{pro:rec_Gauss}
Let $\{X_n, n\geq 1\}$ be a SSG sequence of rvs. For every $n\geq 2$,
let $\calI=\{n\}$, $\calI^\complement=\{1,\ldots,n-1\}$. Then, the probability that $X_n$ is a record and the distribution of $X_n$, given that it is a record,
are equal to
\begin{eqnarray}\label{eq:rec_Gauss}
%
\nonumber \Pro(\isrec_n=1)&=&\Phi_{n-1}(\bzero;\bfGamma_{1:n-1;1:n-1})\label{eq:prob_rec}\\
\nonumber \Pro(X_n\leq x | \isrec_n=1) &=& \Psi_{1,n-1}\left(x;\bfvarrho_{1:n-1},
\bar{\bfSigma}_{1:n-1,1:n-1;n}\right),
\end{eqnarray}
where $\bfGamma_{1:n-1;1:n-1}$ is a $(n-1)\times(n-1)$ variance-covariance matrix whose entries of the 
associated correlation matrix $\bar{\bfGamma}_{1:n-1;1:n-1}$ are 
\begin{equation}\label{eq:par_corr_Gamma}
\gamma_{i,j;n}=\frac{1+\rho_{i,j} -\rho_{i,n}-\rho_{j,n}}{2\sqrt{(1-\rho_{i,n})(1-\rho_{j,n})}}, \; i\neq n, j\neq n
\end{equation}
and $\bar{\bfSigma}_{1:n-1,1:n-1;n}$ is a $(n-1)\times(n-1)$ correlation matrix with entries
\begin{equation*}\label{eq:par_corr}
\rho_{i,j;n}=\frac{\rho_{i,j}-\rho_{i,n}\rho_{j,n}}{\sqrt{(1-\rho^2_{i,n})(1-\rho^2_{j,n})}}, \; i\neq n, j\neq n.
\end{equation*}
\end{prop}
\begin{proof}
%
%
The probability that $X_n$ is a record is
\[
\begin{split}
\Pro(X_n>M_{n-1})&=\int_{-\infty}^{+\infty}\Pro\left(X_i<z,\,\forall \, i\in \calI^\complement \vert X_n=z\right)\phi(z)\diff z\\
&=\int_{-\infty}^{+\infty}\Pro(\bfZ_{1:n-1}\leq z\bfvarrho_{1:n-1})\phi(z)\diff z\\
&=\expect_Z\{\Pro(\bfZ_{1:n-1}\leq Z\bfvarrho_{1:n-1}|Z)\}\\
&=\Pro(\bfZ_{1:n-1}-Z\bfvarrho_{1:n-1}\leq \bfzero)\equiv\Phi_{n-1}(\bzero;\bfGamma_{1:n-1;1:n-1}),
\end{split}
\]
where 
\begin{eqnarray}
\label{eq:std_cov} \bfGamma_{1:n-1;1:n-1}&=&
\bar{\bfSigma}_{1:n-1,1:n-1;n} + \bfvarrho_{1:n-1}\bfvarrho_{1:n-1}^{\top}\\
\label{eq:std_cov_part1} \bfvarrho_{1:n-1}&=&\bfsigma_{1:n-1,1:n-1;n}^{-1}(\bfone_{n-1}-\bar{\bfSigma}_{1:n-1,n})=\left(\sqrt{\frac{1-\rho_{i,n}}{1+\rho_{i,n}}},\forall\, i\in \calI^\complement\right)^{\top}.
\end{eqnarray}
To obtain the second line we used the formula in \eqref{eq:cond_gauss}, which leads to
$\bfZ_{1:n-1}=\bfsigma_{1:n-1,1:n-1;n}^{-1}$
$(\bfX_{1:n-1}-\bfmu_{n})\sim N_{n-1}(\bzero;\bar{\bfSigma}_{1:n-1,1:n-1;n})$, where $\bfmu_{n}=(\rho_{i,n},\forall\, i\in \bcalI)^{\top}v$, and this can be
seen as independent of $Z\sim N(0,1)$. From the third to fourth row we used Lemma~7.1 in \citeN{azzalini1996multivariate}.
With similar steps, we obtain the distribution for the record $X_n$,
\[
\begin{split}
\Pro(X_n\leq x | \isrec_n=1)&=\frac{\Pro(X_n\leq x, X_n>M_{n-1})}{\Pro(X_n>M_{n-1})},\\
&=\frac{\int_{-\infty}^{x}\phi(z)\Phi_{n-1}(z\bfvarrho_{1:n-1};
\bar{\bfSigma}_{1:n-1,1:n-1;n})\diff z}{\Phi_{n-1}(\bzero;\bfGamma_{1:n-1;1:n-1})}\\
&\equiv\Psi_{1,n-1}\left(x;\bfvarrho_{1:n-1},\bar{\bfSigma}_{1:n-1,1:n-1;n}\right).
\end{split}
\]
\end{proof}
The correlations $\rho_{i,j}$, $1\leq i<j\leq n$, in Proposition \ref{pro:rec_Gauss} satisfy $-1\leq \rho_{i,j;n}\leq 1$ \cite{kurowicka2006}  but 
they must also be  such as to satisfy $-1\leq \gamma_{i,j;n}\leq 1$ or 
\[
(\rho_{i,n}+\rho_{j,n}-1)-2\sqrt{(1-\rho_{i,n})(1-\rho_{j,n})} \leq \rho_{i,j} \leq (\rho_{i,n}+\rho_{j,n}-1)+2\sqrt{(1-\rho_{i,n})(1-\rho_{j,n})}.
\]
\begin{rem}\label{ex: independence rec}
Assume in Proposition \ref{pro:rec_Gauss} that $\rho_{i,j}=0$ for all $1\leq i\neq j \leq n$.
Then,
\[
\begin{split}
\Pro\left(\isrec_n=1\right)&=\Phi_{n-1}(\bfzero;\bfI_{n-1}+\bfone_{n-1}\bfone_{n-1}^\top)\\
&=\expect\left(\Phi_{n-1}(\bfone_{n-1} Z;\bfI_{n-1})\right)\\
&=\int_{-\infty}^{+\infty}\Phi_{n-1}(\bfone_{n-1} z;\bfI_{n-1})\phi(z)\diff z=\int_{-\infty}^{+\infty}\Phi^{n-1}(z)\phi(z)\diff z=n^{-1},
\end{split}
\]
where $Z\sim N(0,1)$. 
As expected, we obtain  the results in \cite{gal87} and Lemma 1.1 in \cite{falkkp2018}.
Furthermore,
\[
\begin{split}
\Pro(X_n\leq x| \isrec_n=1)&=\Psi_{1,n-1}\left(x;\bfone_{n-1},\bfI_{n-1}\right)\\
&=n\int_{-\infty}^x\Phi_{n-1}(\bfone_{n-1} z;\bfI_{n-1})\phi(z)\diff z=n\int_{-\infty}^x\Phi^{n-1}(z)\phi(z)\diff z\\
&={\Phi(x)}^n.
\end{split}
\]
\end{rem}

Let
\[
T(k):= \inf\set{m\in\N:\, \sum_{i=1}^m \isrec_i=k}, \quad k\ge 2,\quad T(1):=1,
\]
be the arrival time of the $k$-th record.

\begin{lemma}\label{lemma: joint_times}
Let ${\{T(k)\}}_{k\geq 2}$ be the arrival time process of records. Let $\calI=\{j_2,\dots,j_k\}$ where $2\leq j_2<\dots<j_k\in\N$ and $j_1:=1$. Set $\calI^\complement := \{1,\dots,j_k\}\setminus \calI$.
Then,
\begin{equation*}\label{eq: joint_times}
\begin{split}
&\Pro(T(i)=j_i,i=2,\dots,k)\\
&=\Phi_{j_k-k}(\bfzero;\bfGamma_{\calI^\complement, \calI^\complement})\Psi_{k-1,j_k-k}(\bfzero;\bfD\bar{\bfSigma}_{\calI,\calI}\bfD^\top,\bfDelta,
\bar{\bfSigma}_{\calI^\complement,\calI^\complement;\calI})
\end{split}
\end{equation*}
where $\bfD=(\bfI_{k-1}\quad\bfzero_{k-1})-(\bfzero_{k-1}\quad\bfI_{k-1})$,
\begin{eqnarray}
\label{eq:Dmat} \bfDelta&=&\bfvarrho_{\calI^\complement,\calI^\complement}\bar{\bfSigma}_{\calI,\calI}\bfD^\top{(\bfD\bar{\bfSigma}_{\calI,\calI}\bfD^\top)}^{-1},\\
\label{eq:Gmat}\bfGamma_{\calI^\complement, \calI^\complement}&=&\bfvarrho_{\calI^\complement, \calI^\complement}\bar{\bfSigma}_{\calI,\calI}\bfvarrho_{\calI^\complement, \calI^\complement}^\top+\bar{\bfSigma}_{\calI^\complement,\calI^\complement;\calI},\\
\label{eq:rhomat}\bfvarrho_{\calI^\complement,\calI^\complement}&=&\bfsigma_{\calI^\complement,\calI^\complement;\calI}^{-1}(\bfB-\bfSigma_{\calI^\complement, \calI}
\bar{\bfSigma}_{\calI,\calI}^{-1}),
\end{eqnarray}
and
\begin{equation}\label{eq:Bmat}
\bfB:=
\begin{pmatrix}
\mathbf{1}_{j_2-2} &\bfzero_{j_2-2} & \dots &\bfzero_{j_2-2}&\bfzero_{j_2-2}\\
\bfzero_{j_3-j_2-1} & \mathbf{1}_{j_3-j_2-1} & \dots   &\bfzero_{j_3-j_2-1}&\bfzero_{j_3-j_2-1}\\
\vdots & \vdots &  & \vdots&\vdots\\
\bfzero_{j_k-j_{k-1}-1} & \bfzero_{j_k-j_{k-1}-1} & \dots & \mathbf{1}_{j_k-j_{k-1}-1}& \bfzero_{j_k-j_{k-1}-1}
\end{pmatrix}\in\R^{j_k-k,k-1}
\end{equation}

\end{lemma}
\begin{proof}
We have
\[
\begin{split}
&\Pro(T(i)=j_i,i=2,\dots,k)\\
&=\Pro(M_{j_i+1:j_{i+1}-1} <X_i,i=1,\dots,k-1,X_{j_{k-1}}<X_{j_k})\\
&=\int_{-\infty}^{+\infty}\int_{-\infty}^{z_k}\dots\int_{-\infty}^{z_2}\Pro(M_{j_i+1:j_{i+1}-1}<z_i,i=1,\dots,k-1|X_{j_i}=z_i,i=1,\dots,k-1)\\
&\quad\cdot \phi_k(z_1,\dots,z_k)\diff z_1\dots\diff z_k\\
&=\int_{-\infty}^{+\infty}\int_{-\infty}^{z_k}\dots\int_{-\infty}^{z_2}\Pro(\bfX_{\calI^\complement}<\bfB\bfz | \bfX_{\calI}=\bfz)\phi_k(\bfz;\bar{\bfSigma}_{\calI,\calI})\diff \bfz\\
\end{split}
\]
where $\bfB$ is given in \eqref{eq:Bmat}.
By standardizing the random vector $\bfX_{\bcalI}$, we obtain
\[
\begin{split}
&\int_{-\infty}^{+\infty}\int_{-\infty}^{z_k}\dots\int_{-\infty}^{z_2}
\Phi_{j_k-k}(\bfvarrho_{\calI^\complement,\calI^\complement} \bfz;\bar{\bfSigma}_{\calI^\complement,\calI^\complement;\calI})\phi_k(\bfz;\bar{\bfSigma}_{\calI,\calI})\diff \bfz\\
&=\Phi_{j_k-k}(\bfzero;\bfGamma_{\calI^\complement, \calI^\complement})\int_{-\infty}^{+\infty}\int_{-\infty}^{z_k}\dots\int_{-\infty}^{z_2}\psi_{k,j_k-k}(\bfzero;\bar{\bfSigma}_{\calI,\calI},\bfvarrho_{\calI^\complement,\calI^\complement},\bar{\bfSigma}_{\bfX_{\bcalI},\bfX_{\bcalI};\bfX_{\calI}})\diff \bfz\\
&=\Phi_{j_k-k}(\bfzero;\bfGamma_{\calI^\complement, \calI^\complement})\Pro(Z_1<Z_2<\dots<Z_k)\\
&=\Phi_{j_k-k}(\bfzero;\bfGamma_{\calI^\complement, \calI^\complement})\Pro(Z_1-Z_2<0,\dots,Z_{k-1}-Z_k<0)\\
\end{split}
\]
where $\bfGamma_{\calI^\complement, \calI^\complement}$ and $\bfvarrho_{\calI^\complement,\calI^\complement}$ are
given in \eqref{eq:Gmat} and \eqref{eq:rhomat}.

By recalling formula \eqref{eq: affine CSN},  we obtain
\[
\begin{split}
&
\begin{pmatrix}
Z_1-Z_2\\
\vdots\\
Z_{k-1}-Z_k
\end{pmatrix}
=
\begin{pmatrix}
1 & -1 & 0 & 0 & \dots & 0\\
0&1& -1 & 0 & \dots & 0\\
\dots\\
0 & \dots & 0& 0&1& -1
\end{pmatrix}
\begin{pmatrix}
Z_1\\
\vdots\\
Z_k
\end{pmatrix}\\
&=\bfD\bfZ\sim CSN_{k-1,j_k-k}(\bfD\bar{\bfSigma}_{\calI,\calI}\bfD^\top,\bfDelta,\bar{\bfSigma}_{\calI^\complement,\calI^\complement;\calI})
\end{split}
\]
where $\bfDelta$ is given in \eqref{eq:Dmat}
\end{proof}

In the next result we establish the distribution of the arrival time $T(2)$ of the second record as well as that of the increment $X_{T(2)}-X_1$.

\begin{theorem}\label{teo:time_2nd_rec}
Let $\{X_n, n\geq 1\}$ be a SSG sequence of rvs. Let $\rho_{i,j}=\expect(X_i,X_j)$ with $1\leq i\neq j\leq n$ . Assume that for $n\to\infty$, 
$\rho_{i,j}\to 0$ as $|j-i|\to\infty$ and $\rho_{k,n}\to 1$ as $k\to\infty$. 
For $n=2,3,\ldots$, the distribution of the arrival time of the second record $T(2)$ is
\begin{equation}\label{eq:time_2nd_rec}
\Pro\left(T(2)=n\right)=
\begin{cases}
1/2, & n=2,\\
\Phi_{n-2}(\bzero;\bfGamma_{2:n-1,2:n-1})-
\Phi_{n-1}(\bzero;\bfGamma_{2:n,2:n}), & n> 2
\end{cases}
\end{equation}
where $\bfGamma_{2:n-1,2:n-1}$ and $\bfGamma_{2:n,2:n}$ are defined similarly to \eqref{eq:std_cov}.
Furthermore, for every $x>0$, the distribution of the increment $X_{T(2)}-X_1$ is
\begin{equation}\label{eq:increment_cdf}
H(x)=\sum_{n\geq 2} \Phi_{n-1}\left(\bfu_x;\bfGamma_{2:n,2:n}\right)-\Phi_{n-1}\left(\bfzero;\bfGamma_{2:n,2:n}\right),
%
%
%
%
\end{equation}
where $\bfu_x=(x/(1-\rho^2_{1,n})^{1/2},0,\ldots,0)^\top$ is an $(n-1)$-dimensional vector.
\end{theorem}
\begin{proof}
When $n=2$ we have
\[
\Pro(T(2)=2)=\Pro(X_2>X_1)=1/2.
\]
For $n>2$ we have
\[
\begin{split}
\Pro(T(2)=n)&=\Pro\left( X_i<X_1,\,i=2,\dots,n-1,X_n>X_1\right)\\
&=\Pro\left( X_i<X_1,\,i=2,\dots,n-1\right)-\Pro\left( X_i<X_1,\,i=2,\dots,n\right).
\end{split}
\]
Therefore, \eqref{eq:time_2nd_rec} follows by similar arguments to those used in Proposition \ref{pro:rec_Gauss}.
It must be checked that
\[
\begin{split}
\sum_{n\geq 2} \Pro(T(2)=n)&=\frac{1}{2}+
\lim_{N\to\infty}\sum_{n=3}^N\left(\Phi_{n-2}(\bzero;\bfGamma_{2:n-1,2:n-1})-\Phi_{n-1}(\bzero;\bfGamma_{2:n,2:n})\right)\\
&=\lim_{N\to\infty}\left(1-\Phi_{2}\left(\bzero;\bfGamma_{1:2,1:2}\right)+
\Phi_{2}\left(\bzero;\bfGamma_{1:2,1:2}\right)\right.\\
&\left.-\dots+\Phi_{N-2}\left(\bzero;\bfGamma_{1:N-2,1:N-2}\right)-
\Phi_{N-1}\left(\bzero;\bfGamma_{1:N-1,1:N-1}\right)\right)\\
&=1-\lim_{N\to\infty}\Phi_{N-1}\left(\bzero;\bfGamma_{1:N-1,1:N-1}\right)=1.
\end{split}
\]
Let $(\tilde{X}_1,\ldots,\tilde{X}_{n-1})$ be zero-mean unit-variance Gaussian sequence with variance-covariance matrix
$\bfGamma_{1:n-1,1:n-1}$. 
Set $P_n=\Pro(\tilde{X}_i\leq 0,\ldots,\tilde{X}_{n-1}\leq 0)=\Phi_{n-1}\left(\bzero;\bfGamma_{1:n-1,1:n-1}\right)$.
Clearly $\Phi_{n-1}\left(\bzero;\bfGamma_{1:n-1,1:n-1}\right)=\Phi_{n-1}\left(\bzero;\bar{\bfGamma}_{1:n-1,1:n-1}\right)$.  
We recall that $\Pro(\tilde{X}_i\leq0)=1/2$ for every $i=1,\ldots,n-1$.
 By the Fr\'{e}chet inequalities we have that
$$
A_n:=\max\left(0,\sum_{i=1}^n\Pro(X_i\leq0)-(n-1)\right)=\max(0,1-n/2)\leq P_n \leq 1/2.
$$
For $P_n$ we derive the following upper bound $B_n$. Precisely,
\begin{eqnarray*}
P_n&=&\Pro\left(\sum_{i=1}^{n-1}\indic(\tilde{X}_i\leq 0)\geq n-1\right)=\Pro\left\{\sum_{i=1}^{n-1}\left(\indic(\tilde{X}_i\leq 0)-\frac 12\right)\geq \frac{n-1}{2}\right\}\\
&\leq & \Pro\left\{\left|\sum_{i=1}^{n-1}\left(\indic(\tilde{X}_i\leq 0)-\frac 12\right)\right|\geq \frac{n-1}{2}\right\}\\
&\leq &\frac{4}{(n-2)^2}\expect\left[\left\{ \sum_{i=1}^{n-1}\left(\indic(\tilde{X}_i\leq 0)-\frac 12\right) \right\}^2\right]\\
&=& \frac{4}{(n-2)^2} \sum_{i=1}^{n-1} \sum_{j=1}^{n-1}\cov(\indic(\tilde{X}_i\leq 0), \indic(\tilde{X}_j\leq 0))\\
&=& \frac{4}{(n-2)^2} \sum_{i=1}^{n-1} \sum_{j=1}^{n-1}\cov(P_{i,j;n}-1/4)=: B_n,
\end{eqnarray*}
where $P_{i,j;n}:=\Pro(\tilde{X}_i\leq 0,\tilde{X}_j\leq 0)=\Phi_2(0;\gamma_{i,j;n})$ and where
$\Phi_2(\cdot;\gamma_{i,j;n})$ is a bivariate Gaussian cdf with correlation $\gamma_{i,j;n}$ that
is given in \eqref{eq:par_corr_Gamma}. In the third row we used the Chebyshev's inequality. Set
$h=|j-i|$ we rewrite $B_n$ as
\begin{eqnarray*}\label{eq:B_n}
\nonumber B_n&=& \frac{4}{(n-2)^2} \sum_{h=0}^{n-2}2(n-h)(P_{h;n}-1/4)\\
&=& \frac{8}{n(1+2/n)^2} (P_{0;n}-1/4) +  \frac{8}{n(1+2/n)^2}\sum_{h=1}^{n-2}\left(1-\frac{h}{n}\right)(P_{h;n}-1/4)\\
&=&\alpha_n+\beta_n,
\end{eqnarray*}
where $P_{h;n}:=\Pro(\tilde{X}_0\leq 0,\tilde{X}_h\leq 0)=\Phi_2(0;\gamma_{h;n})$ and
$$
\gamma_{h;n}=\frac{1+\rho_{0,h} -\rho_{0,n-i}-\rho_{h,n-i}}{2\sqrt{(1-\rho_{0,n-i})(1-\rho_{h,n-i})}},\quad h=0,\ldots,n-2.
$$
Now, when $h=0$ we obtain $\gamma_{0;n}=1$ and therefore $P_{0;n}=1/2$ and as a consequence the term 
$\alpha_n\to0$ as $n\to\infty$. We rewrite the term $\beta_n$ as
\begin{eqnarray*}
\beta_n &=& \frac{8}{n(1+2/n)^2}\sum_{h=1}^{n-2}(P_{h;n}-1/4)-  \frac{8}{n(1+2/n)^2}\sum_{h=1}^{n-2}\frac{h}{n}(P_{h;n}-1/4)\\
&=& c_n - d_n.
\end{eqnarray*}
Now by the assumption we have that for $n\to\infty$, $\gamma_{h;n}\to 0$ as $h\to\infty$, therefore for all $\varepsilon>0$
there exists a $n_0$ such that for all $h>n_0$ we have $|P_{h;n}-1/4|<\varepsilon$. As a consequence we have
\begin{eqnarray*}
c_n &=& \frac{8}{n(1+2/n)^2}\left(\sum_{h=1}^{n_0}(P_{h;n}-1/4)+\sum_{h=n_0+1}^{n-2}(P_{h;n}-1/4)\right)\\
&<&  \frac{8}{n(1+2/n)^2} \left(c  + \varepsilon(n-2+n_0+1) \right)=o(1),
\end{eqnarray*}
where $c$ is a positive constant. Therefore, $c_n\to 0$ as $n\to \infty$ and since $d_n<c_n$ then
$\beta_n\to 0$ and $B_n\to0$ as $n\to \infty$. Concluding, since $A_n\leq P_n\leq B_n$ and $A_n=0$ for $n\geq 2$, then $P_n\to0$ as
$n\to\infty$.

Finally, for every $x>0$ the distribution of the increment $X_{T(2)}-X_1$ is
\[
\begin{split}
\Pro(X_{T(2)}-X_1\leq x)&=\sum_{n\geq 2}\Pro(X_n-X_1\leq x,T(2)=n)\\
&=\sum_{n\geq 2}\Pro(X_n-X_1\leq x,X_i<X_1,\,i=2,\dots,n-1,X_n>X_1)\\
&=\sum_{n\geq 2}\Pro(0<X_n-X_1\leq x,X_i<X_1,\,i=2,\dots,n-1)
\end{split}
\]
The term inside the sum is equal to
\[
\begin{split}
&\Pro(0<X_n-X_1\leq x,X_i<X_1,\,i=2,\dots,n-1)\\
&=\int_{-\infty}^{+\infty}\Pro(0<X_n-u\leq x,X_i<u,\,i=2,\dots,n-1\vert X_1=z)\phi(z)\diff z\\
&=\int_{-\infty}^{+\infty}\Pro(X_i<x,\,i=2,\dots,n-1,X_n\leq z+x\vert X_1=z)\phi(z)\diff z\\
&\quad-\int_{-\infty}^{+\infty}\Pro(X_i<z,\,i=2,\dots,n\vert X_1=z)\phi(z)\diff z.
\end{split}
\]
Therefore, \eqref{eq:increment_cdf} follows by similar arguments to those used in Proposition \ref{pro:rec_Gauss}.
\end{proof}
\begin{rem}
Note that when $\rho_{i,j}=0$ for all $1\leq i<j\leq n$ and $n>2$ we obtain
$$
\Pro\left(T(2)=n\right)=\Phi_{n-2}(\bfzero;\bfI_{n-2}+\bfone_{n-2}\bfone_{n-2}^\top)-
\Phi_{n-1}(\bfzero;\bfI_{n-1}+\bfone_{n-1}\bfone_{n-1}^\top)=\frac{1}{n-1}-\frac{1}{n}=\frac{1}{n(n-1)}.
$$
\end{rem}

Let $N:=\sum_{n=1}^\infty \isrec_{n}$ be the number of records among an infinite sequence $X_1,X_2,\ldots$
When the components of the sequence are independent and identically distributed with a continuous df, then it is a well-known result
that an infinite number of records will occur: $E(N)=\sum_{n=1}^\infty P(R_n=1) = \sum_{n=1}^\infty 1/n =\infty$ \cite{gal87}.

A natural question that arises is the following. What is the expected number of records that will take place
in the case of a stationary Gaussian process?
\begin{prop}\label{pro:numb_rec}
Let $\{X_n, n\geq 1\}$ be a SSG sequence of rvs and $\Phi_{n-1}(\bzero;\bfGamma_{1:n-1;1:n-1})$ be the probability that a record take place described in Proposition \ref{pro:rec_Gauss}.
Let $N$ be the number of records among an infinite sequence $X_1,X_2,\ldots$
Then, we have
\[
\expect(N)=
\begin{cases}
%
%
\infty, & \text{if} \quad 1/2\leq \gamma_{i,j;n}\leq 1, \quad \forall\;  1\leq i\neq j < n\\

2, & \text{if} \quad \gamma_{i,j;n}=0 , \quad\;\;\; \qquad \forall\;  1\leq i\neq j < n.
%
\end{cases}
\]
where $\gamma_{i,j;n}$ is the correlation parameter in \eqref{eq:par_corr_Gamma}.
%
%
\end{prop}
\begin{proof}
First, note that
\[
\begin{split}
\expect(N)&=\expect\left(\sum_{n=1}^\infty \isrec_n\right)=\sum_{n=1}^\infty \expect(\isrec_n)\\
&=1+\sum_{n=2}^\infty \Pro(X_{n}>M_{n-1})\\
&=1+\sum_{n=2}^\infty \Phi_{n-1}(\bzero;\bfGamma_{1:n-1,1:n-1}).
\end{split}
\]
The entries of the correlation matrix $\bar{\bfGamma}_{1:n-1,1:n-1}$ in \eqref{eq:par_corr_Gamma} are
 $\gamma_{i,j;n}=1/2$, $1\leq i\neq j<n$, if and only if  $\rho_{i,j}=\rho_{i,n}=\rho_{j,n}=0$.
%
%
%
In this case by Remark \ref{ex: independence rec} we have that $\Phi_{n-1}(\bzero;\bfI_{n-1}+\bfone_{n-1}\bfone_{n-1}^\top)=1/n$.
From this it follows that when $1/2\leq \gamma_{i,j;n}\leq 1$ or 
%
%
\begin{equation}\label{eq:corr_inequaliteis}
\sqrt{(1-\rho_{i,n})(1-\rho_{j,n})}\leq 1+\rho_{i,j} -\rho_{i,n}-\rho_{j,n} \leq 2\sqrt{(1-\rho_{i,n})(1-\rho_{j,n})},
\end{equation}
then $\Phi_{n-1}(\bzero;\bfGamma_{1:n-1,1:n-1})\geq1/n$ and as a consequence
%
%
\[
\expect(N)\geq\sum_{n=1}^{\infty}\frac{1}{n}=\infty.
\]
For every $1\leq i\neq j < n$, provided that $\rho_{i,n}+\rho_{j,n}\geq 0$,
when $\rho_{i,j}=\rho_{i,n}+\rho_{j,n}-1$ then we have  $\bar{\bfGamma}_{1:n-1,1:n-1}=\bfI_{n-1}$.
Therefore in this case $\Phi_{n-1}(\bzero;\bfGamma_{1:n-1,1:n-1})=\Phi_{n-1}(\bfI_{n-1})=2^{-n+1}$.
%
As a consequence
$$
\expect(N)=1+\sum_{n=2}^\infty2^{-n+1}=2\sum_{n=0}^\infty 2^{-n}-2 = 2.
$$
\end{proof}
From Proposition \eqref{pro:numb_rec} it follows that the expected number of records depends on the type of correlation structure of the Gaussian process. For example, an infinite number of records is expected when all variables are uncorrelated or when $X_i$ and $X_j$ are more correlated than the sum of the correlations between $X_i$ and $X_n$, and $X_j$ and $X_n$, for every $1\leq i\neq j < n$.
The second assertion follows from
the left-hand side of the inequality in \eqref{eq:corr_inequaliteis} by noting that
$0\leq\sqrt{(1-\rho_{i,n})(1-\rho_{j,n})}\leq1$.
This suggests
looking at $1+\rho_{i,j} -\rho_{i,n}-\rho_{j,n}\geq 1$ which holds as soon as $\rho_{i,j}\geq\rho_{i,n}+\rho_{j,n}$.
Instead, loosely speaking when $X_i$ and $X_j$ are less correlated than the sum of the correlations between $X_i$ and $X_n$, and $X_j$ and $X_n$, for every $1\leq i\neq j < n$, the expected number of records can be finite. This assertion follows from the condition $\rho_{i,j}=\rho_{i,n}+\rho_{j,n}-1$, provided that $\rho_{i,n}+\rho_{j,n}\geq 0$, which leads that two records should be expected.

In our next result we compute the distribution of the interarrival time between the second and third record.
%
%
\begin{prop}
The distribution of the increment has the representation
\begin{equation}\label{eq: second increment}
\begin{split}
&\Pro(X_{T(3)}-X_{T(2)}\leq x)\\
&= \sum_{j=2}^\infty\sum_{k=j+1}^{\infty}\Phi_{k-3}(\bfzero;\bfGamma_{\calI^\complement, \calI^\complement})
\left\{\Psi_{2,k-3}(\bfzero;\bfD\bar{\bfSigma}_{\calI,\calI}\bfD^\top,\bfDelta,\bar{\bfSigma}_{\calI^\complement,\calI^\complement;\calI})\right.\\
&\hspace*{3cm}\left.-\Psi_{2,k-3}((0,-x);\bfD\bar{\bfSigma}_{\calI,\calI} \bfD^\top,\bfDelta,\bar{\bfSigma}_{\calI^\complement,\calI^\complement;\calI})\right\},
\end{split}
\end{equation}
where the sets of indices $\calI=\{1,j,k\}$ and $\calI^\complement=\{2,\ldots,j-1,j+1,\ldots,k-1\}$ vary with $j$ and $k$, $\bfDelta$ and $\tilde{\bfvarrho}_{\calI^\complement,\calI^\complement}$ are similarly defined as in formula \eqref{eq:Dmat} and \eqref{eq:rhomat} and where
\[
\bfD:=
\begin{pmatrix}
1 & -1 &0\\
0 & 1 &-1
\end{pmatrix}
\]
\end{prop}
\begin{proof}
By the total probability rule
\[
\Pro(X_{T(3)}-X_{T(2)}\leq x)= \sum_{j=2}^\infty\sum_{k=j+1}^{\infty}\Pro(X_k-X_j\leq x, T(3)=k,T(2)=j)
\]
Note that, by repeating the same arguments as the previous proofs
\[
\begin{split}
&\Pro(X_k-X_j\leq x, T(3)=k,T(2)=j)\\
&=\Pro(X_k-X_j\leq x, M_{2:j-1}<X_1,X_1<X_j,M_{j+1:k-1}<X_j,X_j<X_k)\\
&=\int_{-\infty}^{+\infty}\int_{z_k-x}^{z_k}\int_{-\infty}^{z_j}\Pro(M_{2:j-1}<z_1,M_{j+1:k-1}<z_j)\phi(z_1,z_j,z_k)\diff z_1\diff z_j\diff z_k\\
&=\int_{-\infty}^{+\infty}\int_{z_k-x}^{z_k}\int_{-\infty}^{z_j}\Phi_{k-3}(\tilde{\bfvarrho}_{\calI^\complement,\calI^\complement} \bfz;\bar{\bfSigma}_{\calI^\complement,\calI^\complement;\calI})\phi_3(\bfz;\bar{\bfSigma}_{\calI,\calI})\diff\bfz\\
&=\Phi_{k-3}(\bfzero;\bfGamma_{\calI^\complement, \calI^\complement})\Pro(Z_1<Z_j<Z_k)-\Phi_{k-3}(\bfzero;\bfGamma_{\calI^\complement, \calI^\complement})\Pro(Z_1<Z_j<Z_k-x),
\end{split}
\]
and thus, the assertion follows by repeating the arguments in the proof of Lemma \ref{lemma: joint_times}.
\end{proof}

In the following result we derive the probability that two records occur at prescribed indices, with no further record in between, together with the distribution of such consecutive records.

\begin{theorem}\label{theo:joint_cons_rec_Gauss}
Let $\{X_n, n\geq 1\}$ be a SSG sequence of rvs.
For every $n\geq 2$ and $j<n$,
let $\calI=\{j,n\}$, $\calI^\complement=\{1,\ldots,j-1,j+1,\ldots,n-1\}$.
The probability that two consecutive records $X_j$ and $X_n$ occur, is
\begin{equation}\label{eq:joint_cons_rec_Gauss}
\Pro\left(\isrec_j=1, \isrec_n=1,\cap_{i=j+1}^{n-1}\isrec_i=0\right)=
\Phi_{n-2}(\bzero;\bfGamma_{1:n-1\mins j,1:n-1\mins j})-
\Phi_{n-1}(\bzero;\bfGamma_{1:n\mins j,1:n\mins j}),
\\
\end{equation}
where $\bfGamma_{1:n-1\mins j,1:n-1\mins j}$ and $\bfGamma_{1:n\mins j,1:n\mins j}$ are similarly defined as in \eqref{eq:std_cov}.
%
%
The joint distribution of $(X_j,X_n)$, given that they are consecutive records, is
\[
\Pro\left(X_j\leq x_1, X_n \leq x_2 | \isrec_j=1, \isrec_n=1,\cap_{i=j+1}^{n-1}\isrec_i=0\right)=
\begin{cases}
P(x_1,x_2), & x_1\leq x_2\\
P(x_1,x_1) & x_1> x_2
\end{cases}
\]
where
\[
\begin{split}
P(a,b)&=
w_{n-1}(b\bfmu;\tilde{\bfGamma}_{1:n\mins j,1:n\mins j})\Psi_{1,n-1}
\left(a;\tilde{\bfvarrho}_{1:n\mins j},-b\bfmu,\bar{\bfSigma}_{1:n\mins j,1:n\mins j;j}\right) \\
&-
w_{n-1}(\bzero;\bfGamma_{1:n\mins j,1:n\mins j})\Psi_{1,n-1}\left(a;\bfvarrho_{1:n\mins j},\bzero,\bar{\bfSigma}_{1:n\mins j,1:n\mins j;j}\right)
%
%
%
%
%
\end{split}
\]
and where $\bfvarrho_{1:n\mins j}$ is similarly defined as in \eqref{eq:std_cov_part1},
$\tilde{\bfGamma}_{1:n\mins j,1:n\mins j}=\bar{\bfSigma}_{1:n\mins j,1:n\mins j;j}+
\tilde{\bfvarrho}_{1:n\mins j}\tilde{\bfvarrho}_{1:n\mins j}^{\top}$ with
$$
\tilde{\bfvarrho}_{1:n\mins j}=\left(\bfvarrho_{1:n\mins j}^{\top},-\frac{\rho_{n,j}}{\sqrt{1-\rho_{n,j}^2}}\right)^{\top},
$$
\begin{equation*}\label{eq:param}
\bfmu=\left(0,\dots,0, {(1-\rho_{j,n}^2)}^{-1/2}\right)^{\top}\in\R^{n-1}.
\end{equation*}
and for any $\bfx\in\R^{n-1}$ and positive-definite matrix $\bfSigma\in \R^{n-1,n-1}$,
$$
w_{n-1}(\bfx;\bfSigma)=\frac{\Phi_{n-1}(\bfx;\bfSigma)}{\Phi_{n-2}(\bzero;\bfGamma_{1:n-1\mins j, 1:n-1\mins j})-
\Phi_{n-1}(\bzero;\bfGamma_{1:n\mins j,1:n\mins j})}.
$$
\end{theorem}
\begin{proof}
First we compute the probability that two consecutive records occur. For every $1\leq j<n$ we have
\[
\begin{split}
\Pro(\isrec_j=1, \isrec_n=1,\cap_{i=j+1}^{n-1}\isrec_i=0)&=\Pro(X_j>M_{j-1}, X_n>M_{n-1})\\
&=\Pro(X_i<X_j, \forall\, i\in\bcalI, X_n>X_j)\\
&=\Pro(X_i<X_j, \forall\, i\in\bcalI,)\\
&- \Pro(X_i<X_j, \forall\, i\in\{i,\dots,n\}\setminus{\{j\}}).
\end{split}
\]
Therefore, \eqref{eq:joint_cons_rec_Gauss} follows by similar arguments to those used in Proposition \ref{pro:rec_Gauss}.

The joint distribution of $(X_j,X_n)$ is given by
\[
\begin{split}
&\Pro\left(X_j\leq x_1, X_n \leq x_2 | \isrec_j=1, \isrec_n=1,\cap_{i=j+1}^{n-1}\isrec_i=0\right)\\
&=
\frac{\Pro\left(X_j\leq x_1, X_n\leq x_2, X_j>M_{j-1}, X_n>M_{n-1},\cap_{i=j+1}^{n-1}\isrec_i=0\right)}{\Pro(X_j>M_{j-1}, X_n>M_{n-1})}.
\end{split}
\]
Note that
\[
\begin{split}
&\Pro(X_j\leq x_1, X_n\leq x_2, X_j>M_{j-1}, X_n>M_{n-1},\cap_{i=j+1}^{n-1}\isrec_i=0)\\
&=\Pro(X_j\leq x_1, X_n\leq x_2, X_i<X_j,\forall\, i\in\bcalI, X_j<X_n)\\
&=\Pro(X_j\leq x_1, X_n\leq x_2, X_i<X_j,\forall\, i\in\bcalI)\\
&-\Pro(X_j\leq x_1, X_n\leq x_2, X_i<X_j,i=1,\dots,n, i\neq j)\\
&=A(x_1,x_2)-B(x_1,x_2).
\end{split}
\]
When $x_1\leq x_2$, we obtain from similar arguments as those used in the proof of Proposition \ref{pro:rec_Gauss}
\[
\begin{split}
A(x_1,x_2)&:= \Pro(X_j\leq x_1, X_n\leq x_2, X_i<X_j,\forall\, i\in\bcalI)\\
&=\int_{-\infty}^{x_1}\Pro(X_n\leq x_2, X_i<z,\forall\, i\in\bcalI|X_j=z)\phi(z)\diff z\\
&=\int_{-\infty}^{x_1}\Pro\left(Z_i<\sqrt{\frac{1-\rho_{i,j}}{1+\rho_{i,j}}} z,\forall\, i\in\bcalI, Z_n< -\frac{\rho_{n,j}}{\sqrt{1-\rho_{n,j}^2}} z +\frac{x_2}{\sqrt{1-\rho_{n,j}^2}}\right)\phi(z)\diff z\\
&=\int_{-\infty}^{x_1}\Phi_{n-1}(\tilde{\bfvarrho}_{1:n\mins j} z+x_2\bfmu;
\bar{\bfSigma}_{1:n\mins j,1:n\mins j;j}) \phi(z)\diff z\\
&=\Phi_{n-1}(x_2\bfmu;\tilde{\bfGamma}_{1:n\mins j,1:n\mins j})
\Psi_{1,n-1}\left(x_1;\tilde{\bfvarrho}_{1:n\mins j},-x_2\bfmu,
\bar{\bfSigma}_{1:n\mins j,1:n \mins j;j}\right).
\end{split}
\]
Similarly,
\[
\begin{split}
B(x_1,x_2)&:= \Pro(X_j\leq x_1, X_i<X_j,i=1,\dots,n, i\neq j)\\
&=\int_{-\infty}^{x_1}\Pro(X_i<z,i=1,\dots,n, i\neq j|X_j=z)\phi(z)\diff z\\
&=\Phi_{n-1}(\bfzero;\bfGamma_{1:n\mins j,1:n\mins j})\Psi_{1,n-1}\left(x_1;\bfvarrho_{1:n\mins j},\bzero,\bar{\bfSigma}_{1:n\mins j,1:n\mins j;j}\right).
\end{split}
\]
When $x_1> x_2$, it is sufficient to compute $A(x_1,x_2)$ and $B(x_1,x_2)$ in $(x_2,x_2)$.
\end{proof}

In the next result we drop the assumption that the two records in Theorem \ref{theo:joint_cons_rec_Gauss} are consecutive.

\begin{theorem}\label{theo:joint_rec_Gauss}
Let $\{X_n, n\geq 1\}$ be a SSG sequence of rvs.
For every $n\geq 2$ and $j<n$,
let $\calI=\{j,n\}$ and $\calI^\complement=\{1,\ldots,j-1,j+1,\ldots,n-1\}$.
The probability that  $X_j$ and $X_n$ are records, is
\begin{equation*}
\Pro(\isrec_j=1, \isrec_n=1)=\Phi_{n-1}(\bfzero; \tilde{\bfOmega}).
\end{equation*}
The joint distribution of $(X_j,X_n)$, given that they are records, is
$$
\Pro(X_j\leq x_1, X_n\leq x_2\vert \isrec_j=1, \isrec_n=1)=
\begin{cases}
P(x_1,x_2), & x_1\leq x_2\\
P(x_1,x_1) & x_1> x_2
\end{cases}
$$
where
\begin{equation*}\label{eq:joint_distr_Gauss2}
\begin{split}
P(a,b)&=
\frac{\Phi_{n-2}(\bfzero;\bfGamma_{\bcalI,\bcalI})}{\Phi_{n-1}(\bfzero; \tilde{\bfOmega})}
\Big(\Psi_{2,n-2}(a,b; \bar{\bfSigma}_{\calI,\calI},\bfvarrho_{\bcalI, {\bcalI}},\bar{\bfSigma}_{{\bcalI},{\bcalI};\calI})\Bigg.
-\Psi_{2,n-2}(a,a; \bar{\bfSigma}_{\calI,\calI},\bfvarrho_{\bcalI, {\bcalI}},\bar{\bfSigma}_{{\bcalI},{\bcalI};\calI})\\
&\hspace*{4cm}\Big.+\Psi_{2,n-2}(0,a;\bfD\bar{\bfSigma}_{{\calI},{\calI}}\bfD^\top,
\bfDelta,\bar{\bfSigma}^{**}_{{\bcalI},{\bcalI};\calI})\Big).
\end{split}
\end{equation*}
%
%
\end{theorem}
\begin{proof}
Similar steps as those used in the proof of Proposition \ref{pro:rec_Gauss} show
that the probability that $X_j$ and $X_j$ are records, is
\[
\begin{split}
&\Pro(\isrec_j=1, \isrec_n=1)=\Pro(X_j>M_{j-1}, X_n>M_{n-1})\\
&=\int_{-\infty}^{+\infty}\int_{-\infty}^{z_2}\Pro\left( M_{j-1}<z_1, M_{j+1}^{n-1}<z_2 \vert X_j=z_1, X_n=z_2\right)\phi_2(z_1,z_2;\bar{\bfSigma}_{\calI,\calI})\diff z_1\diff z_2\\
&=\int_{-\infty}^{+\infty}\int_{-\infty}^{z_2}\Phi_{n-2}\left(\bfvarrho_{\bcalI, {\bcalI}}\bfz;\bar{\bfSigma}_{{\bcalI},{\bcalI};\calI}\right)\phi_2(\bfz;\bar{\bfSigma}_{\calI,\calI})\diff z_1\diff z_2\\
&=\Phi_{n-2}(\bfzero, \bfGamma_{\calI,\bcalI})\int_{-\infty}^{+\infty}\int_{-\infty}^{z_2}\psi_{2,n-2}\left(\bfz;\bar{\bfSigma}_{\calI,\calI},\bfvarrho_{\bcalI, {\bcalI}},\bar{\bfSigma}_{{\bcalI},{\bcalI};\calI}\right)\diff z_1\diff z_2\\
&=\Phi_{n-2}(\bfzero;\bfGamma_{\calI,\bcalI})\Pro(Z_1-Z_2<0),
\end{split}
\]
where $(Z_1,Z_2)\sim CSN_{2,n-2}(\bar{\bfSigma}_{\calI,\calI},\bfvarrho_{\bcalI, {\bcalI}},\bar{\bfSigma}_{{\bcalI},{\bcalI};\calI})$.
Precisely, to obtain the third line we used the formula in \eqref{eq:cond_gauss} and where
\begin{align}
\notag\bfB&=
\begin{pmatrix}
\bfone_{j-1} & \bfzero_{j-1}\\
\bfzero_{n-j-1} & \bfone_{n-j-1}
\end{pmatrix},\\
\notag\bfvarrho_{\bcalI, {\bcalI}}&=\bfsigma_{{\bcalI},{\bcalI};\calI}^{-1}(\bfB-\bfSigma_{{\bcalI}, {\calI}}\bar{\bfSigma}_{{\calI}, {\calI}}^{-1})\\
&=
\label{eq:varrhomat}\begin{pmatrix}
\frac{1}{\sigma_{11}}\left(1-\frac{\rho_{1j}-\rho_{1n}\rho_{jn}}{1-\rho_{jn}^2}\right) & \frac{\rho_{1j}\rho_{jn}-\rho_{1n}}{\sigma_{11}(1-\rho_{jn}^2)}\\
\vdots & \vdots\\
\frac{1}{\sigma_{j-1,j-1}}\left(1-\frac{\rho_{j-1,j}-\rho_{j-1,n}\rho_{jn}}{1-\rho_{jn}^2}\right) & \frac{\rho_{j-1,j}\rho_{jn}-\rho_{j-1,n}}{\sigma_{j-1,j-1}(1-\rho_{jn}^2)}\\
\frac{\rho_{j+1,j}\rho_{jn}-\rho_{j+1,n}}{\sigma_{j+1,j+1}(1-\rho_{jn}^2)} & \frac{1}{\sigma_{j+1,j+1}}\left(1-\frac{\rho_{j+1,j}-\rho_{j+1,n}\rho_{jn}}{1-\rho_{jn}^2}\right)\\
\vdots & \vdots\\
\frac{\rho_{n-1,j}\rho_{jn}-\rho_{n-1,n}}{\sigma_{n-1,n-1}(1-\rho_{jn}^2)} & \frac{1}{\sigma_{n-1,n-1}}\left(1-\frac{\rho_{n-1,j}-\rho_{n-1,n}\rho_{jn}}{1-\rho_{jn}^2}\right)
\end{pmatrix}
\end{align}
and $\bar{\bfSigma}_{{\bcalI},{\bcalI};\calI}$ is a $(n-2)\times(n-2)$ partial correlation matrix with upper diagonal entries
\begin{equation*}\label{eq:par_corr_2}
\rho_{i,k;j,n}=\rho_{ij}-\frac{\rho_{ij}-\rho_{in}\rho_{jn}}{1-\rho_{jn}^2}\rho_{kj}-\frac{\rho_{in}-\rho_{ij}\rho_{jn}}{1-\rho_{jn}^2}\rho_{kn}, \; \forall \,i<k\in \bcalI.
\end{equation*}
and
\[
\sigma_{i,i}=1-\frac{\rho_{ij}-\rho_{in}\rho_{jn}}{1-\rho_{jn}^2}\rho_{ij}-\frac{\rho_{in}-\rho_{ij}\rho_{jn}}{1-\rho_{jn}^2}\rho_{in}.
\]
In the third line we multiply and divide the term within the integrals with $\Phi_{n-2}(\bfzero;\bfGamma_{\bcalI, {\bcalI}})$ where $\bfGamma_{\bcalI, {\bcalI}}$ is defined as
\[
\bfGamma_{\bcalI, {\bcalI}}=\bfvarrho_{\bcalI, {\bcalI}}\bar{\bfSigma}_{\calI,\calI}\bfvarrho_{\bcalI, {\bcalI}}^\top+\bar{\bfSigma}_{{\bcalI},{\bcalI};\calI}.
\]
We therefore recognize a unified multivariate skew-normal pdf within the integrals and the integral
of it can be seen as $\Pro(Z_1<Z_2)$. Now, by \eqref{eq: affine CSN} we obtain
\[
Z_1-Z_2=
\begin{pmatrix}
1 & -1
\end{pmatrix}
\begin{pmatrix}
Z_1\\
Z_2
\end{pmatrix}
\sim CSN_{1,n-2}\left(2(1-\rho_{j,n}),\bfDelta^*,\bfSigma^*_{\bcalI, {\bcalI};\calI}\right)
\]
where
\[
\begin{split}
\bfDelta^*&=\frac{1}{2(1-\rho_{j,n})}\bfvarrho_{\bcalI, {\bcalI}}\bar{\bfSigma}_{\calI,\calI}
\begin{pmatrix}
1 \\ -1
\end{pmatrix}\\
&=\frac{1}{2}\begin{pmatrix}
\mathbf{1}_{j-1} \\ -\mathbf{1}_{n-j-1}
\end{pmatrix}
{\left(\frac{1}{\sigma_{ii}}\left(1+\frac{\rho_{in}-\rho_{ij}}{1-\rho_{jn}}\right)\right)}_{i=1,\dots,n-1,i\neq j}
\end{split}
\]
and
\[
\bfSigma^*_{\bcalI, {\bcalI};\calI}=\bfGamma_{\bcalI, {\bcalI}}-\bfDelta^*
\begin{pmatrix}
1-\rho_{jn} & -1+\rho_{jn}
\end{pmatrix}
\bfvarrho_{\bcalI, {\bcalI}}^\top.
\]
By formula \eqref{eq: CSN cdf} we obtain the result, with
\[
\tilde{\bfOmega}=
\begin{pmatrix}
\bfGamma_{\bcalI, {\bcalI}} & 2(1-\rho_{jn}){\bfDelta^*}^\top\\
2(1-\rho_{jn}){\bfDelta^*} & 2(1-\rho_{jn})
\end{pmatrix}.
\]
By similar steps we can compute the joint distribution of two records $(X_j,X_n)$ for $j<n$.
\[
\Pro(X_j\leq x_1, X_n \leq x_2 | \isrec_j=1, \isrec_n=1)=
\frac{\Pro(X_j\leq x_1, X_n\leq x_2, X_j>M_{j-1}, X_n>M_{n-1})}{\Pro(X_j>M_{j-1}, X_n>M_{n-1})}.
\]
The numerator can be written as
\[
\begin{split}
&\Pro(X_j\leq x_1, X_n\leq x_2,X_j>M_{j-1}, X_n>M_{n-1})\\
&=\int_{-\infty}^{x_2}\int_{-\infty}^{\min(x_1,z_2)}\Pro\left( M_{j-1}<z_1, M_{j+1}^{n-1}<z_2 \vert X_j=z_1, X_n=z_2\right)\phi(\bfz;\bar{\bfSigma}_{{\calI},\calI})\diff z_1\diff z_2\\
&=\int_{-\infty}^{x_2}\int_{-\infty}^{x_1}\Pro\left( M_{j-1}<z_1, M_{j+1}^{n-1}<z_2 \vert X_j=z_1, X_n=z_2\right)\phi_2(\bfz;\bar{\bfSigma}_{{\calI},\calI})\indic(z_2>x_1)\diff z_1\diff z_2\\
&+\int_{-\infty}^{x_2}\int_{-\infty}^{z_2}\Pro\left( M_{j-1}<z_1, M_{j+1}^{n-1}<z_2 \vert X_j=z_1, X_n=z_2\right)\phi_2(\bfz;\bar{\bfSigma}_{{\calI},\calI})\indic(z_2<x_1)\diff z_1\diff z_2\\
&=\int_{x_1}^{x_2}\int_{-\infty}^{x_1}\Pro\left( M_{j-1}<z_1, M_{j+1}^{n-1}<z_2 \vert X_j=z_1, X_n=z_2\right)\phi_2(\bfz;\bar{\bfSigma}_{{\calI},\calI})\diff z_1\diff z_2\\
&+\int_{-\infty}^{x_1}\int_{-\infty}^{z_2}\Pro\left( M_{j-1}<z_1, M_{j+1}^{n-1}<z_2 \vert X_j=z_1, X_n=z_2\right)\phi_2(\bfz;\bar{\bfSigma}_{{\calI},\calI})\diff z_1\diff z_2\\
&=\int_{-\infty}^{x_2}\int_{-\infty}^{x_1}\Phi_{n-2}\left(\bfvarrho_{\bcalI, {\bcalI}}\bfz;\bar{\bfSigma}_{{\bcalI},{\bcalI};\calI}\right)\phi_2(\bfz;\bar{\bfSigma}_{{\calI},\calI})\diff z_1\diff z_2\\
&-\int_{-\infty}^{x_1}\int_{-\infty}^{x_1}\Phi_{n-2}\left(\bfvarrho_{\bcalI, {\bcalI}}\bfz;\bar{\bfSigma}_{{\bcalI},{\bcalI};\calI}\right)\phi_2(\bfz;\bar{\bfSigma}_{{\calI},\calI})\diff z_1\diff z_2\\
&+\int_{-\infty}^{x_1}\int_{-\infty}^{z_2}\Phi_{n-2}\left(\bfvarrho_{\bcalI, {\bcalI}}\bfz;\bar{\bfSigma}_{{\bcalI},{\bcalI};\calI}\right)\phi_2(\bfz;\bar{\bfSigma}_{{\calI},\calI})\diff z_1\diff z_2\\
\end{split}
\]
where $\bfvarrho_{\bcalI, {\bcalI}}$ is as in \eqref{eq:varrhomat}. We multiply and divide each
term within the integrals with $\Phi_{n-2}(\bfzero;\bfGamma_{\bcalI, {\bcalI}})$.
Then, we recognize that the first two integrals provide the distribution of the closed skew-normal random vector we introduced before,
%
%
evaluated at the points $(x_1,x_2)$, $(x_1,x_1)$. Instead, the third integral represents the distribution
of the random vector $(Z_1-Z_2, Z_1)$ which again according to \eqref{eq: affine CSN} follows
a closed skew-normal distribution, i.e.,
\[
\begin{split}
&
\begin{pmatrix}
1 & -1\\
0&1
\end{pmatrix}
\begin{pmatrix}
Z_1\\
Z_2
\end{pmatrix}
=\bfD\begin{pmatrix}
Z_1\\
Z_2
\end{pmatrix}
\sim CSN_{2,n-2}(\bfD\bar{\bfSigma}_{{\calI},{\calI}}\bfD^\top,\bfDelta,\bar{\bfSigma}^{**}_{{\bcalI},{\bcalI};\calI}),
\end{split}
\]
where $\bfDelta=\bfvarrho_{\bcalI, {\bcalI}}\bar{\bfSigma}_{{\bcalI},{\bcalI}}\bfD^\top{(\bfD\bar{\bfSigma}_{{\calI},{\calI}}\bfD^\top)}^{-1}$ and $\bar{\bfSigma}^{**}_{{\bcalI},{\bcalI};\calI}=\bfGamma_{\bcalI, {\bcalI}}-\bfDelta\bfD\bar{\bfSigma}_{{\calI},{\calI}}\bfvarrho_{\bcalI,\bcalI}^\top$.
\end{proof}

It follows from Theorem \ref{theo:joint_rec_Gauss}  that the two events:  a record occuring  at time $j$ and $n$, are not independent. Indeed, the probability
$\Phi_{n-1}(\bfzero; \tilde{\bfOmega})$ is different from the product of the two marginal probabilities
$\Phi_{j-1}(\bzero;\bfGamma_{1:j-1;1:j-1})$ and $\Phi_{n-1}(\bzero;\bfGamma_{1:n-1;1:n-1})$,
derived in Proposition \ref{pro:rec_Gauss}.

\begin{rem}
The marginal distribution of $X_j$, given that $(X_j,X_n)$ are records, is
\[
\begin{split}
&\Pro(X_j\leq x_1\vert \isrec_j=1, \isrec_n=1)\\
&=\frac{\Phi_{n-2}(\bfzero;\bfGamma_{\bcalI,\bcalI})}{\Phi_{n-1}(\bfzero; \tilde{\bfOmega})}
\times\left(\Psi_{1,n-2}(x_1; 1,\bfDelta_1,\bar{\bfOmega}_{{\bcalI},{\bcalI};\calI})\right.
-\Psi_{2,n-2}(x_1,x_1; \bar{\bfSigma}_{\calI,\calI},\bfvarrho_{\bcalI, {\bcalI}},\bar{\bfSigma}_{{\bcalI},{\bcalI};\calI})\\
&\hspace*{4cm}\left.+\Psi_{2,n-2}(0,x_1;\bfD\bar{\bfSigma}_{{\calI},{\calI}}\bfD^\top,
\bfDelta,\bar{\bfSigma}_{{\bcalI},{\bcalI};\calI})\right).
\end{split}
\]
where
\[
\bfDelta_1=
\begin{pmatrix}
{\sigma_{ii}^{-1}(1-\rho_{ij})}_{i=1,\dots,j-1}\\
{\sigma_{ii}^{-1}(\rho_{jn}-\rho_{in})}_{i=j+1,\dots,n-1}
\end{pmatrix},
\]
\[
\bar{\bfOmega}_{{\bcalI},{\bcalI};\calI}=\bar{\bfSigma}_{{\bcalI},{\bcalI};\calI}+
\begin{pmatrix}
{\sigma_{ii}^{-1}(\rho_{ij}\rho_{jn}-\rho_{in})}_{i=1,\dots,j-1}\\
{\sigma_{ii}^{-1}(1-\rho^2_{jn}-\rho_{ij}-\rho_{in}\rho_{jn})}_{i=j+1,\dots,n-1}
\end{pmatrix},
\]
and these parameters are obtained from \eqref{eq: affine CSN}, with $\bfA:=(0\quad 1)$.
See the proof of Theorem \ref{theo:joint_rec_Gauss} for the details.
%
Hence, similarly to the case of independent random variables in \shortciteN{falkkp2018},
the distribution of $X_j$ being a record is affected, if we know
that $X_n$ is a record as well.
The marginal distribution of $X_n$, given that $(X_j,X_n)$ are records, is
\[
\Pro(X_n\leq x_2\vert \isrec_j=1, \isrec_n=1)=
\frac{\Phi_{n-2}(\bfzero;\bfGamma_{\bcalI,\bcalI})}{\Phi_{n-1}(\bfzero; \tilde{\bfOmega})}
\Psi_{2,n-2}(0,x_2;\bfD\bar{\bfSigma}_{{\calI},{\calI}}\bfD^\top,
\bfDelta,\bar{\bfSigma}_{{\bcalI},{\bcalI};\calI})\\
\]
Hence, different to the case of independent random variables in \shortciteN{falkkp2018}
we have that the distribution of $X_n$, being a record, is affected by the additional knowledge that at time $j<n$ there was a record.
\end{rem}

\subsection{Asymptotic results for records of stationary sequences}\label{sec:asymptotic}

Although stationary Gaussian sequences are useful for a wide range of statistical analysis
(e.g., \citeNP{lindgren2012}, \citeNP{brockwell2013}, \citeNP{banerjee2014}, \citeNP{cressie2015}, to name a few), a natural question that arises is the following.
What are the properties of records for a stationary sequence of dependent rvs when the univariate marginal distribution, $F$, is non-Gaussian? This question is even more relevant if it is assumed that $F$ is unknown, which concerns many real-world applications. Some of the previous results are clearly independent of the underlying df $F$, provided it is continuous. The probability that $X_n$ is a record, or the distribution of the arrival time of the $n$-th record, for example, do not depend on $F$. The distribution of $X_n$, conditional to the assumption that it is a record, however does depend on $F$.

Let $\{X_n, n\geq 1\}$ be a strictly stationary sequence of rvs, i.e.
the joint distribution of $(X_{j_1},\ldots,X_{j_n})$ and $(X_{j_1+m},\ldots,X_{j_n+m})$ are
identical, for every $n,m$ and $j_1,\ldots,j_n$.
We provide an answer to the above question  under some restrictions on the tail behavior of the marginal distribution of such a process and on  the dependence structure.
Precisely, we assume that $F$ belongs to the (maximum) domain of attraction of $G_\gamma$,
in symbols $F\in\calD(G_\gamma)$, $\gamma\in\R$. This means that, if $Y_1,\ldots,Y_n$ are iid rv
with common cdf $F$, then there exist sequences of norming constants $a_n>0$ and $b_n\in\R$ such that
\begin{equation}\label{eq:Doa}
\lim_{n\to\infty}\Pro(\max(Y_1,\ldots,Y_n)\leq a_n x+b_n)=\lim_{n\to\infty}F^n(a_n x+b_n)=G_\gamma(x),\quad x,\gamma\in\R.
\end{equation}
This cdf is the Generalized Extreme-Value (GEV) class of distributions.
The cdfs of the three sub-classes of the GEV, i.e. the Gumbel, Fr\'{e}chet, and negative Weibull
are denoted by $G_0(x)$, $G_\alpha(x)=G_{1/\gamma}((x-1)/\gamma)$ for $\gamma>0$ and $G_\beta(x)=G_{-1/\gamma}(-(x+1)/\gamma)$
for $\gamma<0$ (see e.g., Ch. 2 \citeNP{fahure10} for details).

Concerning the dependence structure of $\{X_n,n\geq 1\}$ we assume a mild condition on the long-range dependence of extremes
of such a stationary sequence.
Precisely,
we assume that $\{X_n,n\geq 1\}$ is a stationary sequence with a univariate marginal df $F$ that satisfies
$F\in \calD(G_\gamma)$, $\gamma\in\R$ and according to \shortciteN[Ch. 3]{lealr83} the following dependence restriction is required.
Partition $\{1,\ldots,n\}$ into $k_n=\floor*{n/r_n}$ blocks of length $r_n=o(n)$.
Suppose that for every $\lambda>0$, there is a sequence of real-value thresholds $u_n(\lambda)$, $n=1,2\ldots$, such that
\[
\lim_{n\to\infty}n\Pro(X_1>u_n(\lambda))=\lambda,\quad \lambda>0
\]
and the condition $D(u_n(\lambda))$ is satisfied for each such $\lambda$. Specifically, for every
$\lambda>0$, let
\[
K_n(l)=\max(|\Pro(X_i\leq u_n(\lambda), i\in I\cup J)-\Pro(X_i\leq u_n(\lambda), i\in I)\Pro(X_i\leq u_n(\lambda), i\in J)|)
\]
where $I,J\subset \{1,\ldots,n\}$ such that
$\min\{|i-j|:i\in I,j\in J\}=l$. Then, we say that condition $D(u_n(\lambda))$ holds for each such $\lambda$, if
$K_n(l_n)\to 0$ as $n\to\infty$ for some sequence $l_n\to\infty$ with $l_n=o(n)$
(pp. 53-57, \shortciteNP{lealr83}). By Lemma 3.2.2 in \shortciteN{lealr83} this means that extreme events,
such as the partial maxima $M_{E_i}=\max_{j\in E_i}(X_j)$, with $E_i=\{(i-1)r_n+1,\ldots,ir_n\}\setminus\{i r_n-l_n+1,\ldots,i r_n\}$ , $i=1,\ldots,k_n$, which are separated by $l_n$, are almost independent.

Then, under these conditions
by \shortciteN[Theorem 3.7.1]{lealr83} we have that for suitable choices of
$r_n\to\infty$ with $n\to\infty$ such that
$k_nK_n(l_n)\to 0$ and $k_nl_n\to 0$ as $n\to\infty$, it follows that
\begin{equation}\label{eq:approx}
\lim_{n\to\infty}\Pro(\max(X_1,\ldots,X_n)\leq u_n(\lambda))=\exp(-\theta \lambda),\quad 0<\theta\leq1,\,\lambda>0.
\end{equation}
When this holds true we say that the sequence $\{X_n, n\geq 1\}$ has {\it extremal index} $\theta\in(0,1]$.
The result in \eqref{eq:approx} implies that for $\lambda=-\log G_\gamma(x)$, $x\in\R$, and
suitable norming constants $a_n>0$ and $b_n\in\R$,
$$
\lim_{n\to\infty}\Pro\left(\max(X_1,\ldots,X_n)\leq a_n x +b_n\right)=G_\gamma(x)^\theta,\quad x,\gamma\in\R,
$$

Loosely speaking, the extremal index is a parameter that quantifies the impact that the dependence structure of  the stationary sequence has on the asymptotic distribution of extreme events such as the partial maximum $M_{n}$, for sufficiently large $n$. When $\theta=1$ we recover \eqref{eq:Doa}, i.e. the asymptotic distribution of the normalized maximum for a sequence of independent variables.
When $\theta<1$, then for every $x\in\R$ we have that $G_\gamma(x)\leq G^{\theta}_\gamma(x)$ and
therefore $1-G_\gamma(x)\geq 1-G^{\theta}_\gamma(x)$. In other words,
the dependence of the stationary sequence reduces the size of the extreme events.

\begin{theorem}\label{theo:uni_rec_ext_index}
Let $\{X_n,n\geq 1\}$ be a stationary sequence that has extremal index $0< \theta \leq 1$.
Then,
\begin{equation}\label{eq:rec_ext_index}
n\Pro(\isrec_n=1)\to \theta^{-1},\quad \text{as} \quad n\to\infty.
\end{equation}
Furthermore, there are sequences of norming constants $a_n>0$ and
$b_n\in\R$ such that the asymptotic distribution of $X_n$ (suitably normalized), given that it is a record, is
\begin{equation}\label{eq:rec_cdf_ext_index}
\lim_{n\to\infty}\Pro(X_n\leq a_nx +b_n| \isrec_n=1)=G^\theta_\gamma(x),\quad x,\gamma\in\R,\;0<\theta\leq 1.
\end{equation}
\end{theorem}
\begin{proof}
First, we show that there are on average approximately $\theta^{-1}$ records among
$X_1,\ldots,X_n$, for large $n$. Precisely, \eqref{eq:rec_ext_index} is obtained from
\begin{equation*}
\begin{split}
n\Pro(X_n>M_{n-1})&=n\int_{\text{supp}(F)}\Pro(M_{n-1}<v|X_n=v)f_{X_n}(v)\diff u\\
&=\int\underbrace{\Pro(M_{n-1}<u_n(t)|X_n=u_n(t))}_{A_n}\times \underbrace{na_n f_{X_n}(u_n(t))
\indic(\D)}_{B_n} \diff t
%
\end{split}
\end{equation*}
where we used the change of variable $v=u_n(t)$ with $u_n(t)=a_n t + b_n$ and
$\D:=\{t\in\R:u_n(t)\in\text{supp}(F)\}$.
%
By Theorem 3.7.1 in \shortciteNP{lealr83}, for any $t\in\D$, we have 
\begin{equation*}
\begin{split}
\Pro(M_{n-1}\leq u_n(t),X_n\leq u_n(t))&=\Pro(M_{n-1}\leq u_n(t)|X_n\leq u_n(t))(1-\Pro(X_n>u_n(t)))\\
&=\Pro(M_{n-1}\leq u_n(t)|X_n\leq u_n(t))(1+o(1)),
\end{split}
\end{equation*}
and, on the other hand, we have
\begin{equation*}
\begin{split}
\Pro(M_{n-1}\leq u_n(t))&\geq \Pro(M_{n-1}\leq u_n(t), X_n\leq u_n(t))\\
&\geq\Pro(M_{n-1}\leq u_n(t)) - \Pro(M_{n-1}\leq u_n(t), X_n> u_n(t))\\
&\geq\Pro(M_{n-1}\leq u_n(t)) - \Pro(X_n> u_n(t))\\
&\geq\Pro(M_{n-1}\leq u_n(t)) + o(1).
\end{split}
\end{equation*}
From these two results it follows that for any $t\in\D$ we have
\begin{equation}\label{eq:independence}
\Pro(M_{n-1}\leq u_n(t)|X_n\leq u_n(t))=\Pro(M_{n-1}\leq u_n(t))+ o(1).
\end{equation}
By \eqref{eq:independence} and Theorem 3.7.1 in \shortciteNP{lealr83} (cf. \citeNP{o1987extreme}, \citeNP{rootzen1988maxima}) we have
\begin{equation*}
\begin{split}
A_n&\approx \Pro(M_{n-1}<u_n(t)),\quad n\to\infty\\
&\approx\exp\left\{-n\Pro(X_1>u_n(t))\Pro\left(M_{r_n}\leq u_n(t)|X_1>u_n(t)\right)\right\},\quad \text{as }n\to \infty.
\end{split}
\end{equation*}
Note that
\[
n\Pro(X_1>u_n(t))\stackrel{n\to\infty}{\longrightarrow} V(t)=
\begin{cases}
t^{-\alpha},\, t>0,\,\alpha<0, & \text{if } F\in \calD(G_\alpha),\\
(-t)^{\beta},\, t<0,\, \beta>0, & \text{if } F\in\calD(G_\beta),\\
e^{-t},\,t\in\R, & \text{if } F\in\calD (G_0),
\end{cases}
\]
and
\[
\Pro\left(M_{r_n}\leq u_n(t))|X_1>u_n(t)\right) \stackrel{n\to\infty}{\longrightarrow} \theta.
\]
By \citeN[Ch. 2]{resn87} we have
\[
B_n\stackrel{n\to\infty}{\longrightarrow} g(t)\indic(t\in\text{supp}(G_\gamma))=
\begin{cases}
\alpha t^{-(\alpha+1)},\, t>0,\,\alpha<0, & \text{if } F\in \calD(G_\alpha),\\
\beta (-t)^{\beta-1},\, t<0,\, \beta>0, & \text{if } F\in\calD(G_\beta),\\
e^{-t},\,t\in\R, & \text{if } F\in\calD (G_0).
\end{cases}
\]
Therefore, putting all these results together we obtain, as $n\to\infty$,
\begin{equation*}
\begin{split}
n\Pro(X_n>M_{n-1})&\approx \int_{\text{supp}(G_\gamma)}\exp\left\{-V(t)\theta\right\}g(t)\diff t 
=\theta^{-1},
\end{split}
\end{equation*}
and hence, \eqref{eq:rec_ext_index} is proven.

Finally, using similar arguments we obtain
\begin{equation*}
\begin{split}
\Pro(X_n\leq a_nx +b_n| \isrec_n=1)&=\frac{\Pro(X_n\leq a_nx +b_n, X_n>M_{n-1})}{\Pro(X_n>M_{n-1})}\\
&=\frac{n\int_{v\in\text{supp}(F):v\leq a_nx +b_n}\Pro(M_{n-1}<v|X_n=v)f_{X_n}(v)\diff u}{n\Pro(X_n>M_{n-1})}\\
&\stackrel{n\to\infty}{\longrightarrow}G^{\theta}_{\gamma}(x),\quad x,\gamma\in\R
\end{split}
\end{equation*}
and the proof is complete.
\end{proof}

Theorem \ref{theo:uni_rec_ext_index} states that for a stationary sequence of dependent rvs
$\{X_n,n\geq1\}$, under appropriate conditions on the dependence structure, the asymptotic distribution of $X_n$ (appropriately normalized), being a record, coincides with the asymptotic distribution $G_\gamma^\theta$ of the normalized maximum. This finding generalizes Lemma 2.1 in \citeN{falkkp2018}, derived for a sequence of indepedent rvs. Indeed, the same result is obtained for $\theta=1$.

In the following part the are three specific examples of asymptotic distributions of records that stem from the general formula  \eqref{eq:rec_cdf_ext_index} in Theorem \ref{theo:uni_rec_ext_index}.
\begin{example}[\citeNP{chernick1981limit}]
For an integer $m\ge 2$, let $\{\varepsilon_n, n\geq 1\}$ be a sequence of iid rvs
uniformly distributed on $\{0,1/m,\ldots,(m-1)/m\}$. Let $X_0$ be a rv uniformly distributed on
$[0,1]$, being independent of $\{\varepsilon_n\}$.
The process
\[
X_n=m^{-1}X_{n-1}+\varepsilon_n,\quad n\geq 1,
\]
defines a strictly stationary first-order autoregressive sequence.
For $n=1,2,\ldots$ take the norming constants $a_n>1/n$ and $b_n=1$. Then,
\[
\lim_{n\to\infty}\Pro(X_n\leq 1+ x/n | \isrec_n=1)=e^{\theta\, x},\quad x<0,
\]
where $\theta=(m-1)/m$ with $m\geq 2$.
\end{example}
\begin{example}[\shortciteNP{hsing1996extremes}]
Let $\{X_{n,i}, n\geq 1, i\geq 0\}$ be a triangular array of rvs such that
for every $n$ $\{X_{n,i}, i\geq 0\}$ is a SSG sequence.
Define
$\rho_{n,j}=\expect(X_{n,i}\,X_{n,i+j})$ with $i\leq n$ and $j\geq 1$.
Assume that $(1-\rho_{n,j})\log n\to \delta_j\in(0,\infty]$ for all $j\geq 1$ as $n\to\infty$.
For $n=1,2,\ldots$ choose the norming constants $a_n=(2\log n)^{-1/2}$ and
$$
b_n=a_n^{-1} -a_n(\log\log n + \log 4\pi)/2
$$
Then,
\[
\lim_{n\to\infty}\Pro(X_n\leq a_nx +b_n| \isrec_n=1)=e^{-\theta e^{-x}},\quad x\in\R,
\]
where
\[
\theta=\expect_U\left\{\Phi_{|K|}\left(\sqrt{\delta_k}-\frac{U}{2\sqrt{\delta_k}};\bfSigma\right)\right\}
\]
and where $K=\{k\in A\subset \{1,2\ldots\}:\delta_k< \infty\}$, $U$ is a standard exponential rv and
$\bfSigma$ is a correlation matrix with upper diagonal entries
\begin{equation*}
\frac{\delta_{i}+\delta_{j}-\delta_{|i-j|}}{2\sqrt{\delta_{i}\delta_{j}}}, \quad 1\leq i<j\leq |K|.
\end{equation*}
\end{example}
\begin{example}[\shortciteNP{lealr83} Ch. 3.8]
Let $\varepsilon_1,\varepsilon_2,\ldots$ be iid stable $(1,\alpha,\kappa)$ rvs. We recall that a
rv is stable $(\tau,\alpha,\kappa)$ with $\tau\leq 0$, $0<\alpha\leq 2$ and
$|\kappa|\leq 1$ if its characteristic function is
$$
\omega(x)=\exp\left\{-\tau^\alpha |x|^\alpha\left(1-i\frac{\kappa h(x,\alpha)x}{|x|}\right)\right\}
$$
where $i^2=-1$ and $h(x,\alpha)=\tan(\pi\alpha/2)$ for $\alpha\neq 1$ and $h(x,1)=2\pi^{-1}\log |x|$
otherwise. Let $\{c_i, i\in\mathbb{Z}\}$ be a sequence of constants satisfying
$\sum_{i=-\infty}^\infty |c_i|^{\alpha}<\infty$ and $\sum_{i=-\infty}^\infty c_i\log|c_i|$ is
convergent for $\alpha=1$ and $\kappa\neq 0$. Define the moving average process
\[
X_n=\sum_{i=-\infty}^{\infty} c_i\varepsilon_{n-i},\quad n\geq 1.
\]
%
%
For $n=1,2,\ldots$ choose the  norming constants $a_n=n^{1/\alpha}$ and $b_n=0$.
Then,
\[
\lim_{n\to\infty}\Pro(X_n\leq xn^{1/\alpha}| \isrec_n=1)=e^{-\theta x^{-\alpha}},\quad x>0,
\]
where
\[
\theta=k_\alpha(c_+^{\alpha}(1+\kappa)+c_{-}^{\alpha}(1-\kappa)),
\]
with $c_{\pm}=\max_{-\infty<i<\infty}c_i^{\pm}$, $c^{\pm}=\max(0,\pm c_i)$ and $k_\alpha=\pi^{-1}\Gamma(\alpha)\sin(\alpha\pi/2)$.
\end{example}

What is the expected number of records that will take place
in the case of a stationary sequence of rvs that have extremal index $0< \theta \leq 1$?
We know that
\[
\Pro(X_{n}>M_{n-1})\approx\frac{1}{n\theta},\quad n\to\infty.
\]
Therefore, by elementary arguments,
\[
\expect(N)=\sum_{n=1}^\infty \Pro(X_{n}>M_{n-1})
\stackrel{n\to\infty}{\longrightarrow}\infty.
\]

\section{Records of dependent multivariate Gaussian sequences}\label{sec:multivariate}

Let $\{\bfX_n,n\geq 1\}$ be a sequence of $d$-dimensional random vectrors $\bfX_n=(X_n^{(1)},\ldots,X_n^{(d)})\in\R^d$.
We recall that the rv $\bfX_n$ is a complete record if
$
\bfX_n> \max(\bfX_1,\ldots,\bfX_{n-1})
$
where the maximum is computed componentwise.
Here we consider a second-order stationary multivariate Gaussian process and
we extend some of the results derived in Section \ref{sec:uni_gauss_seq} to the multivariate case.
Precisely, we study the probability that a complete record $\bfX_n$ occurs and the distribution of $\bfX_ n$ (being a record). We also study the probability that
two complete records $(\bfX_j,\bfX_ n)$ occur and the joint distribution of
$(\bfX_j,\bfX_ n)$ (being records).
Without loss of generality, assume for simplicity that
$\expect(\bfX_i)=0$, $\expect(\bfX^2_i)=1$ for every $1\leq i\leq n$.

Let $\bfX=(\bfX_1,\ldots,\bfX_n)$ be
an $nd$-dimensional random vector and consider the partition
$\bfX=(\bfX_{\calI}^{\top},\bfX_{\calI^\complement}^{\top})^{\top}\sim N_{nd}(\bfmu,\bfSigma)$ with corresponding partition of the parameters $\bfmu$ and
$\bfSigma$. The formula of the conditional distribution of $\bfX_{\calI^\complement}$ given that $\bfX_{\calI}=\bfx_{\calI}$, for all $\bfx_{\calI}\in\R^{|{\calI}|}$, in \eqref{eq:cond_gauss} is still valid with the obvious changes. Further on  we will provide the specific details whenever we use such a formula.

%

%
\begin{prop}\label{pro:rec_Gauss}
Let $\{\bfX_n, n\geq 1\}$ be a SSG sequence of random vectors in $\R^d$. For every $n\geq 2$, the probability that $\bfX_n$ is a record and the distribution of $\bfX_n$, given that it is a record,
are equal to
\begin{align}
&\Pro(\iscrec_n=1)=\Phi_{(n-1)d}(\bzero; \bfGamma_{{1:n-1},n})\label{eq: prob_Crec_SG},\\
&\Pro(\bfX_n\leq \bfx | \iscrec_n=1) = \Psi_{d,(n-1)d}\left(\bfx; \bar{\bfSigma}_{n},\bfvarrho_{{1:n-1},n},\bar{\bfSigma}_{{1:n-1},{1:n-1};n}\right),
\end{align}
where
\begin{align}
\nonumber\bfB&=
\begin{pmatrix}
\bfone_{n-1} &\bfzero_{n-1} &\dots&\bfzero_{n-1} \\
\bfzero_{n-1} &\bfone_{n-1} &\dots&\bfzero_{n-1} \\
\vdots & \vdots & &\vdots\\
\bfzero_{n-1} &\bfzero_{n-1} &\dots&\bfone_{n-1} \\
\end{pmatrix}\in\R^{(n-1)d,d}\\
\bfvarrho_{{1:n-1},n}&=\bfsigma_{{1:n-1},{1:n-1};n}^{-1}(\bfB-\bfSigma_{{1:n-1},n}^{(1:d)}\bar{\bfSigma}_{n}^{-1})\in\R^{(n-1)d,d}\label{eq: rho_CR}
\end{align}
where $\bfSigma_{{1:n-1},n}^{(1:d)}$ is the covariance matrix of $(X_1^{(1)},X_2^{(1)},\dots,X_{n-1}^{(1)},\dots,X_1^{(d)},X_2^{(d)},\dots,X_{n-1}^{(d)})$ and $\bfX_n$,
$\bar{\bfSigma_n}$ is the variance-covariance matrix of $\bfX_n$
and
\[
\bfGamma_{{1:n-1},n}=\bar{\bfSigma}_{{1:n-1},{1:n-1};n} + \bfvarrho_{{1:n-1},n}\bar{\bfSigma}_{n}\bfvarrho_{{1:n-1},n}^\top
\]
\end{prop}
\begin{proof}
We start deriving the probability that $\bfX_n$ is a record.
\[
\begin{split}
\Pro(\bfX_n>\bfM_{n-1})&=\Pro(X_{n,i}>M_{n-1,i},i=1,\dots,d)\\
&=\int_{\R^d}\Pro(M_{n-1,i}<z_i ,i=1,\dots,d| X_{n,i}=z_i,i=1,\dots,d)\phi_d(\bfz)\diff \bfz
\end{split}
\]
Let $\bfX_{1:n-1}^{(i)}$ be the vector of the $i$-th components of $\bfX_1,\dots,\bfX_{n-1}$. Then,
\[
\begin{split}
&\Pro(M_{n-1,i}<z_i ,i=1,\dots,d| X_{n,i}=z_i,i=1,\dots,d)\\
&=\Pro(\cap_{i=1}^d\{X_{k,i}<z_i ,k=1,\dots,n-1\}| X_{n,i}=z_i,i=1,\dots,d)\\
&=\Pro(\bfX_{1:n-1}^{(i)}<\bfone_{n-1}z_i,i=1,\dots,d | \bfX_n=\bfz) .	
\end{split}
\]

By the multivariate version of the conditional Gaussian distribution in \eqref{eq:cond_gauss} we
have
\[
(\bfX_1,\bfX_2,\dots,\bfX_{n-1} | \bfX_n=\bfz_n)\sim N_{(n-1)d}\left(\bfmu_{1:n-1,n},\bfSigma_{{1:n-1,1:n-1};n}\right)
\]
where
\begin{align*}
&\bfmu_{1:n-1,n}={\left(\bfSigma_{{1:n-1},n}^{(i)}\bar{\bfSigma}_{n}^{-1}\right)}_{i=1,\dots,d}\bfz\in\R^{(n-1)d}\\
&\bfSigma_{{1:n-1},{1:n-1};n}={\left(\bfSigma_{{1:n-1},{1:n-1}}^{(i,h)}-\bfSigma_{{1:n-1},n}^{(i)}\bar{\bfSigma}_{n}^{-1}{\bfSigma^{(h)}}_{{1:n-1},n}^\top\right)}_{i,h=1,\dots,d}.
\end{align*}
$\bfSigma_{1:n-1,n}^{(i)}$ is the covariance matrix of $\bfX_{1:n-1}^{(i)}$ and $\bfX_n$, and $\bfSigma_{{1:n-1},{1:n-1}}^{(i,h)}$ is the covariance matrix of $\bfX_{1:n-1}^{(i)}$ and $\bfX_{1:n-1}^{(h)}$.
We have that
\[
\Pro(\bfX_{1:n-1}^{(i)}<\bfone_{n-1}z_i,i=1,\dots,d | \bfX_n=\bfz)=\Phi_{(n-1)d}(\bfvarrho_{{1:n-1},n}\bfz)	
\]
where
$\bfvarrho_{{1:n-1},n}$ is defined as in equation \eqref{eq: rho_CR}.
Therefore
\[
\begin{split}
\Pro(\bfX_n>\bfM_{n-1})&=\int_{\R^d}\Phi_{(n-1)d}(\bfvarrho_{{1:n-1},n}\bfz)\phi_d(\bfz)\diff \bfz\\
&=\expect\left(\Phi_{(n-1)d}(\bfvarrho_{{1:n-1},n}\bfZ)\right),
\end{split}
\]
where $\bfZ\sim N_{(n-1)d}(\bfzero,\bar{\bfSigma}_{{1:n-1},{1:n-1};n})$ and the claim follows by applying Proposition 7.1  in \cite{azzca99}.

The computation of the distribution function follows the same procedure. We need to compute
\[
\begin{split}
\Pro(\bfX_n \leq \bfx,\bfX_n>\bfM_{n-1})&=\int_{(-\mathbb{\infty},\bfx]}\Pro(M_{n-1,i}<z_i ,i=1,\dots,d | X_{n,i}=z_i,i=1,\dots,d)\phi_d(\bfz)\diff \bfz\\
&=\Phi_{(n-1)d}(\bfzero;\bar{\bfSigma}_{{1:n-1};n}+\bfvarrho_{{1:n-1},n}\bar{\bfSigma}_{n}\bfvarrho_{{1:n-1},n}^\top)\\
&\times\Psi_{d,(n-1)d}(\bfx;\bar{\bfSigma}_{n},\bfvarrho_{{1:n-1},n},\bar{\bfSigma}_{{1:n-1};n})
\end{split}
\]
\end{proof}
\begin{rem}
For every $n=1,2\ldots$, $m<n$ and $1\leq i<j\leq d$, if $\cor(X_n^{(i)},X_n^{(j)})=0$
and $\cor(X_n^{(i)},X_m^{(i)})=0$, then we obtain
\[
\Pro\left( \bfX_n>\bfM_{n-1}\right)=\Phi_{(n-1)d}(\bfzero;\bar{\bfSigma}_{{1:n-1},{1:n-1};n}+\bfvarrho_{{1:n-1},n}\bfvarrho_{{1:n-1},n}^\top)
\]
where the variance-covariance matrix is a diagonal block matrix, with each diagonal block being equal to $\bfI_{n-1}+\bfone_{n-1}\bfone_{n-1}^\top$.
Since we have $d$ diagonal blocks, we obtain
\[
\Phi_{(n-1)d}(\bfzero;\bar{\bfSigma}_{{1:n-1},{1:n-1};n}+\bfvarrho_{{1:n-1},n}\bfvarrho_{{1:n-1},n}^\top)={\left(\Phi_{n-1}(\bfzero;\bfI_{n-1}+\bfone_{n-1}\bfone_{n-1}^\top)\right)}^d=n^{-d}
\]
by the result in Remark \ref{ex: independence rec}.
\end{rem}
Our next result deals with the joint distribution of two complete records at times $j$ and $n>j$. In the following, we use the notation $\bfx_{\calJ},\,\calJ\subseteq\{1,\dots,d\}$ to indicate a vector of dimension $\abs{\calJ}$ whose components are the entries of $\bfx\in\R^d$ determined by the elements of $\calJ$.
\begin{theorem}\label{teo: joint_CR_SG}
Let $\{\bfX_n, n\geq 1\}$ be a SSG sequence of random vectors in $\R^d$. For $j$ and $n>j$, set $\calI=\{j,n\}$. Then
\begin{align}
&\Pro(\iscrec_j=1,\iscrec_n=1)=\Phi_{d(n-2)}(\bzero;\bfGamma_{{\bcalI}{\bcalI}})\Psi_{d,d(n-2)}\left(\bfzero;\mathcal{L}_1\right)\label{eq: joint_CR_SG}\\
&\Pro(\bfX_j\leq \bfx_1,\,\bfX_n\leq \bfx_2 | \iscrec_j=1, \iscrec_n=1)\notag\\
& = \frac{\sum_{\calJ\subseteq \{1,\dots,d\}
}\Psi_{2d,d(n-2)}\left(\bfzero_{\calJ},\bfx_{1\bar{\calJ}},\bfx_{1\calJ},\bfx_{2\bar{\calJ}};\mathcal{L}_{\calJ}\right)-\Psi_{2d,d(n-2)}\left(\bfzero_{\calJ},\bfx_{1\bar{\calJ}},\bfx_{1\calJ},\bfx_{1\bar{\calJ}};\mathcal{L}_{\calJ}\right)}{\Psi_{d,d(n-2)}\left(\bfzero;\mathcal{L}_1\right)}\label{eq: joint_distr_CR_SG},
\end{align}
where $\Psi_{m,q}(\cdot;\mathcal{L})\sim\text{CSN}_{m,q}(\mathcal{L})$,
$\bfGamma_{{\bcalI}{\bcalI}}:=\bar{\bfSigma}_{{\bcalI}{\bcalI};{\calI}}+\bfvarrho_{{1:n-1},n}\bar{\bfSigma}_{{\calI}{\calI}}\bfvarrho_{{1:n-1},n}^\top$,
\begin{equation}\label{eq: CSN L}
\begin{split}
\mathcal{L}_{(\cdot)}&=\left(\bfD_{(\cdot)}\bar{\bfSigma}_{{\calI}{\calI}}\bfD_{(\cdot)}^\top,\bfDelta_{(\cdot)},\bar{\bfSigma}_{{\bcalI}{\bcalI};{\calI}}\right),
\end{split}
\end{equation}
$\bfDelta_{(\cdot)}=\bfvarrho_{{1:n-1},n}\bar{\bfSigma}_{{\bcalI},{\bcalI}}\bfD_{(\cdot)}^\top{(\bfD_{(\cdot)}\bar{\bfSigma}_{{\calI},{\calI}}\bfD_{(\cdot)}^\top)}^{-1}$ and
\[
D_1=
\begin{pmatrix}
\bfI_d &-\bfI_d
\end{pmatrix}\qquad
D_{\calJ}=
\begin{pmatrix}
\bfI_{\calJ} & \bfzero_{\bar{\calJ}} &-\bfI_{\calJ} & \bfzero_{\bar{\calJ}}\\
 \bfzero_{\calJ}& \bfI_{\bar{\calJ}}& \bfzero_{\calJ}& \bfzero_{\bar{\calJ}}\\
 \bfzero_{\calJ}& \bfzero_{\bar{\calJ}}& \bfI_{\calJ}& \bfzero_{\bar{\calJ}}\\
 \bfzero_{\calJ}& \bfzero_{\bar{\calJ}}& \bfzero_{\calJ}& \bfI_{\bar{\calJ}}
\end{pmatrix}
\]
\end{theorem}
\begin{proof}
We compute
\[
\begin{split}
&\Pro(\iscrec_j=1,\iscrec_n=1)\\
&=\Pro(\bfX_j>\bfM_{j-1},\bfX_n>\bfM_{n-1})\\
&=\Pro(\bfX_j>\bfM_{j-1},\bfX_n>\bfM_{j+1:n-1},\bfX_j<\bfX_n)\\
&=\int_{\R^d}\int_{(-\bfinfty,\bfz_n]}\Pro(\bfM_{j-1}<\bfz_j,\bfM_{j+1:n-1}<\bfz_n | \bfX_i=\bfz_i,i\in\calI)\phi_{2d}(\bfz_j,\bfz_n;\bar{\bfSigma}_{{\calI}{\calI}})\diff \bfz_j\diff \bfz_n\\
\end{split}
\]
First of all, we recall the inverse blok-matrix of a two-by-two block matrix:
\[
\bar{\bfSigma}_{{\calI}{\calI}}^{-1}=
\begin{pmatrix}
\bfLambda_1 & -\bar{\bfSigma}_{j}^{-1}\bfSigma_{j,n}\bfLambda_2\\
-\bfLambda_2\bfSigma_{n,j}\bar{\bfSigma}_{j}^{-1} & \bfLambda_2
\end{pmatrix}
\]
where $\bfLambda_i$ is the Schur complement of $\bar{\bfSigma}_{i}$ in $\bar{\bfSigma}_{{\calI}{\calI}}$, for $i\in\calI$, for example, $\bfLambda_1={\left(\bar{\bfSigma}_{j}-\bfSigma_{j,n}\bar{\bfSigma}_{n}^{-1}\bfSigma_{n,j}\right)}^{-1}$.
By the multivariate version of the conditional Gaussian distribution in \eqref{eq:cond_gauss} we
have that
$(\bfX_i,i\in\bcalI)|(\bfX_i=\bfz_i,i\in\calI)\sim N_{\abs{\bcalI}}\left(\bfmu_{{\bcalI}{\bcalI};{\calI}},\bfSigma_{{\bcalI}{\bcalI};{\calI}}\right)$. Specifically we have
$\bfmu_{{\bcalI}{\bcalI};{\calI}}={\left({(\bfmu_i^\top,i=1,\dots,d)},
{(\bfmu_h^\top,h=d+1,\dots,2d)}\right)}^\top$ with
\begin{equation}
\begin{split}
\bfmu_i&=\left(\bfSigma_{{1:j-1},j}^{(i)}\bfLambda_1-\bfSigma_{{1:j-1},n}^{(i)}\bfLambda_2\bfSigma_{n,j}\bar{\bfSigma}_{j}^{-1}\right)\bfz_j+\left(-\bfSigma_{{1:j-1},j}^{(i)}\bar{\bfSigma}_{j}^{-1}\bfSigma_{j,n}\bfLambda_2+\bfSigma_{{1:j-1},n}^{(i)}\bfLambda_2\right)\bfz_n\\
&=:\bfmu_{ij}\bfz_j+\bfmu_{in}\bfz_n,
\end{split}
\end{equation}
\begin{equation}
\begin{split}
\bfmu_h&=\left(\bfSigma_{{j+1:n-1},j}^{(h-d)}\bfLambda_1-\bfSigma_{{j+1:n-1},n}^{(h-d)}\bfLambda_2\bfSigma_{n,j}\bar{\bfSigma}_{j}^{-1}\right)\bfz_j
+\left(-\bfSigma_{{j+1:n-1},j}^{(h-d)}\bar{\bfSigma}_{j}^{-1}\bfSigma_{j,n}\bfLambda_2+\bfSigma_{{j+1:n-1},n}^{(h-d)}\bfLambda_2\right)\bfz_n\\
&=:\bfmu_{hj}\bfz_j+\bfmu_{hn}\bfz_n,
\end{split}
\end{equation}
and $\bfSigma_{{\bcalI}{\bcalI};{\calI}}$ is a $(2d)\times(2d)$ matrix. It is defined by blocks ${(\bfSigma_{{\bcalI}{\bcalI};{\calI}})}_{i,h=1,\dots,2d}$ of the form
\[
\begin{split}
&{(\bfSigma_{{\bcalI}{\bcalI};{\calI}})}_{ih}\\
&=\bfSigma_{{1:j-1},{1:j-1}}^{(i,h)}-{\left(\left(\bfSigma_{{1:j-1},j}^{(i)}\bfLambda_1-\bfSigma_{{1:j-1},n}^{(i)}\bfLambda_2\bfSigma_{n,j}\bar{\bfSigma}_{j}^{-1}\right)\bfSigma_{{1:j-1},j}^{(h)}\right.} \\
&{\left. \quad+
\left(-\bfSigma_{{1:j-1},j}^{(i)}\bar{\bfSigma}_{j}^{-1}\bfSigma_{j,n}\bfLambda_2+\bfSigma_{{1:j-1},n}^{(i)}\bfLambda_2\right)\bfSigma_{{1:j-1},n}^{(h)}\right)}_{ih}\qquad i,h=1,\dots,d
\end{split}
\]
\[
\begin{split}
&{(\bfSigma_{{\bcalI}{\bcalI};{\calI}})}_{ih}\\
&=\bfSigma_{{1:j-1},{j+1:n-1}}^{(i,h-d)}-{\left(\left(\bfSigma_{{1:j-1},j}^{(i)}\bfLambda_1-\bfSigma_{{1:j-1},n}^{(i)}\bfLambda_2\bfSigma_{n,j}\bar{\bfSigma}_{j}^{-1}\right)\bfSigma_{{j+1:n-1},j}^{(h-d)}\right.} \\
&{\left. \quad+
\left(-\bfSigma_{{1:j-1},j}^{(i)}\bar{\bfSigma}_{j}^{-1}\bfSigma_{j,n}\bfLambda_2+\bfSigma_{{1:j-1},n}^{(i)}\bfLambda_2\right)
\bfSigma_{{j+1:n-1},n}^{(h-d)}\right)}_{ih}\qquad i=1,\dots,d,\, h=d+1,\dots,2d
\end{split}
\]
\[
\begin{split}
&{(\bfSigma_{{\bcalI}{\bcalI};{\calI}})}_{ih}\\
&=\bfSigma_{{j+1:n-1},{1:j-1}}^{(i-d,h)}-{\left(\left(\bfSigma_{{j+1:n-1},j}^{(i-d)}\bfLambda_1-\bfSigma_{{j+1:n-1},n}^{(i-d)}\bfLambda_2\bfSigma_{n,j}\bar{\bfSigma}_{j}^{-1}\right)\bfSigma_{{1:j-1},j}^{(h)}\right.} \\
&{\left. \quad+
\left(-\bfSigma_{{j+1:n-1},j}^{(i-d)}\bar{\bfSigma}_{j}^{-1}\bfSigma_{j,n}\bfLambda_2+\bfSigma_{{j+1:n-1},j}^{(i-d)}\bfLambda_2\right)\bfSigma_{{1:j-1},n}^{(h)}\right)}_{ih}\qquad i=d+1,\dots,2d,\,h=1,\dots,d
\end{split}
\]
\[
\begin{split}
&{(\bfSigma_{{\bcalI}{\bcalI};{\calI}})}_{ih}\\
&=\bfSigma_{{j+1:n-1},{j+1:n-1}}^{(i-d,h-d)}-{\left(\left(\bfSigma_{{j+1:n-1},j}^{(i-d)}\bfLambda_1-\bfSigma_{{j+1:n-1},n}^{(i-d)}\bfLambda_2\bfSigma_{n,j}\bar{\bfSigma}_{j}^{-1}\right)\bfSigma_{{j+1:n-1},j}^{(h-d)}\right.} \\
&{\left. \quad+
\left(-\bfSigma_{{j+1:n-1},j}^{(i-d)}\bar{\bfSigma}_{j}^{-1}\bfSigma_{j,n}\bfLambda_2+\bfSigma_{{j+1:n-1},j}^{(i-d)}\bfLambda_2\right)\bfSigma_{{j+1:n-1},n}^{(h-d)}\right)}_{ih}\qquad i,h=d+1,\dots,2d
\end{split}
\]
Therefore, we obtain
\[
\Pro(\bfX_{1:j-1}^{(i)}<\bfone_{j-1}\bfz_j,\bfX_{j+1:n-1}^{(i)}<\bfone_{n-j-1}\bfz_n,i=1,\dots,d |\bfX_i=\bfz_i,i\in\calI)=\Pro(\bfZ<\bfvarrho_{\bcalI,\bcalI}\bfz)
\]
where $\bfZ\sim N_{(n-2)d}(\bfzero,\bar{\bfSigma}_{{\bcalI}{\bcalI};{\calI}})$,
\begin{equation}\label{eq: coeff_block_B}
\bfvarrho_{\bcalI,\bcalI}=\bfsigma_{\bcalI,\bcalI;\calI}^{-1}(\bfB-\bfSigma_{\bcalI,\calI}^{(1:d)}\bar{\bfSigma}_{\calI\calI}^{-1})\in\R^{(n-2)d,d}
\end{equation}
\[
\bfB=
\begin{pmatrix}
\bfone_{j-1} & \bfzero_{j-1} & \dots & \bfzero_{j-1} &\bfzero_{j-1} &\bfzero_{j-1} &\dots&\bfzero_{j-1} \\
\bfzero_{j-1} & \bfone_{j-1} & \dots & \bfzero_{j-1} &\bfzero_{j-1} &\bfzero_{j-1} &\dots&\bfzero_{j-1} \\
\vdots & \vdots & &\vdots &\vdots &\vdots & &\vdots\\
\bfzero_{j-1} & \bfzero_{j-1} & \dots & \bfone_{j-1} &\bfzero_{j-1} &\bfzero_{j-1} &\dots&\bfzero_{j-1} \\
\bfzero_{n-j-1} & \bfzero_{n-j-1} & \dots & \bfzero_{n-j-1} &\bfone_{n-j-1} &\bfzero_{n-j-1} &\dots&\bfzero_{n-j-1} \\
\bfzero_{n-j-1} & \bfzero_{n-j-1} & \dots & \bfzero_{n-j-1} &\bfzero_{n-j-1} &\bfone_{n-j-1} &\dots&\bfzero_{n-j-1} \\
\vdots & \vdots & &\vdots &\vdots &\vdots & &\vdots\\
\bfzero_{n-j-1} & \bfzero_{n-j-1} & \dots & \bfzero_{n-j-1} &\bfzero_{n-j-1} &\bfzero_{n-j-1} &\dots&\bfone_{n-j-1} \\
\end{pmatrix}
\]
and $\bfz=(\bfz_j,\bfz_n)$ is a column vector with length $2d$.
We obtain
\[
\begin{split}
&\Pro(\iscrec_j=1,\iscrec_n=1)\\
&=\int_{\R^d}\int_{(-\bfinfty,\bfz_n]}\Phi_{d(n-2)}(\bfvarrho_{\bcalI,\bcalI}\bfz;\bar{\bfSigma}_{{\bcalI}{\bcalI};{\calI}})\phi_{2d}(\bfz_j,\bfz_n;\bar{\bfSigma}_{{\calI}{\calI}})\diff \bfz_j\diff \bfz_n\\
&=\Phi_{d(n-2)}(\bzero; \bar{\bfSigma}_{{\bcalI}{\bcalI};{\calI}}+\bfvarrho_{\bcalI,\bcalI}\bar{\bfSigma}_{{\calI}{\calI}}\bfvarrho_{\bcalI,\bcalI}^\top)\Pro(\bfZ_1<\bfZ_2)
\end{split}
\]
where $(\bfZ_1,\bfZ_2)\sim \text{CSN}_{2d,d(n-2)}(\bar{\bfSigma}_{{\calI}{\calI}},\bfvarrho_{\bcalI,\bcalI},\bar{\bfSigma}_{{\bcalI}{\bcalI};{\calI}})$. Formula \eqref{eq: joint_CR_SG} follows by noting that
\[
\bfZ_1-\bfZ_2=
\begin{pmatrix}
\bfI_d &-\bfI_d
\end{pmatrix}
\begin{pmatrix}
\bfZ_1\\
\bfZ_2
\end{pmatrix}
\]
and by applying \eqref{eq: affine CSN}

To compute formula \eqref{eq: joint_distr_CR_SG}, we repeat the same procedure. With $\mathcal{D}(\bfa_1,\dots,\bfa_n):=\prod (-\infty,\bfa_i]$ we obtain
\[
\begin{split}
&\Pro(\bfX_j\leq \bfx_1,\,\bfX_n\leq \bfx_2 , \iscrec_j=1, \iscrec_n=1)\\
&=\Pro(\bfX_j\leq \bfx_1,\,\bfX_n\leq \bfx_2 ,\bfX_j>\bfM_{j-1},\bfX_n>\bfM_{j+1:n-1}\bfX_i,\bfX_j<\bfX_n)\\
&=\int_{(-\bfinfty,\bfx_2]}\sum_{\calJ\subseteq\{1,\dots,d\}}\left(\int_{\mathcal{D}(\bfz_{n\calJ},\bfx_{1\bar{\calJ}})}\hspace*{-1cm}\Pro(\bfM_{j-1}<\bfz_j,\bfM_{j+1:n-1}<\bfz_n | \bfX_i=\bfz_i,i\in\calI)\phi_{2d}(\bfz_j,\bfz_n;\bar{\bfSigma}_{{\calI}{\calI}})\diff \bfz_j\right)\\
&\hspace*{1cm}\indic(\bfz_{n\calJ}<\bfx_{1\calJ},\bfz_{n\bar{\calJ}}>\bfx_{1\bar{\calJ}})\diff \bfz_n\\
&=\sum_{\calJ\subseteq\{1,\dots,d\}}\int_{(-\bfinfty,\bfx_{1\calJ}]}\int_{(\bfx_{1\calJ},\bfx_{2\bar{\calJ}}]}\int_{\mathcal{D}(\bfz_{n\calJ},\bfx_{1\bar{\calJ}})}\Phi_{d(n-2)}(\bfvarrho_{\bcalI,\bcalI}\bfz;\bar{\bfSigma}_{{\bcalI}{\bcalI};{\calI}})\phi_{2d}(\bfz_j,\bfz_n;\bar{\bfSigma}_{{\calI}{\calI}})\diff \bfz_j\diff \bfz_n\\
&=\sum_{\calJ\subseteq\{1,\dots,d\}}\left(\int_{\mathcal{D}(\bfx_{1\calJ},\bfx_{2\bar{\calJ}},\bfz_{n\calJ},\bfx_{1\bar{\calJ}})}\Phi_{d(n-2)}(\bfvarrho_{\bcalI,\bcalI}\bfz;\bar{\bfSigma}_{{\bcalI}{\bcalI};{\calI}})\phi_{2d}(\bfz_j,\bfz_n;\bar{\bfSigma}_{{\calI}{\calI}})\diff \bfz_j\diff \bfz_n\right.\\
&\hspace*{0.5cm}-\left.\int_{\mathcal{D}(\bfx_{1\calJ},\bfx_{2\bar{\calJ}},\bfz_{n\calJ},\bfx_{1\bar{\calJ}})}\Phi_{d(n-2)}(\bfvarrho_{\bcalI,\bcalI}\bfz;\bar{\bfSigma}_{{\bcalI}{\bcalI};{\calI}})\phi_{2d}(\bfz_j,\bfz_n;\bar{\bfSigma}_{{\calI}{\calI}})\diff \bfz_j\diff \bfz_n\right)\\
&=\sum_{\calJ\subseteq \{1,\dots,d\}
}\Phi_{d(n-2)}(\bzero; \bar{\bfSigma}_{{\bcalI}{\bcalI};{\calI}}+\bfvarrho_{\bcalI,\bcalI}\bar{\bfSigma}_{{\calI}{\calI}}\bfvarrho_{\bcalI,\bcalI}^\top)\\
&\hspace{2cm}\left(\Pro(\bfZ_{1\calJ}<\bfZ_{2\calJ},\bfZ_{1\bar{\calJ}}\leq \bfx_{1\bar{\calJ}},\bfZ_{2\calJ}\leq \bfx_{1\bar{\calJ}},\bfZ_{2\bar{\calJ}}\leq \bfx_{2\bar{\calJ}})\right.\\
&\hspace{2cm}-\left.\Pro(\bfZ_{1\calJ}<\bfZ_{2\calJ},\bfZ_{1\bar{\calJ}}\leq \bfx_{1\bar{\calJ}},\bfZ_{2\calJ}\leq \bfx_{1\bar{\calJ}},\bfZ_{2\bar{\calJ}}\leq \bfx_{1\bar{\calJ}})\right)
\end{split}
\]
The first probability on the right-hand side can be computed by noting that
\[
\begin{pmatrix}
\bfZ_{1\calJ}-\bfZ_{1\calJ}\\
\bfZ_{1\bar{\calJ}}\\
\bfZ_{2\calJ}\\
\bfZ_{2\bar{\calJ}}
\end{pmatrix}=
\begin{pmatrix}
\bfI_{\calJ} & \bfzero_{\bar{\calJ}} &-\bfI_{\calJ} & \bfzero_{\bar{\calJ}}\\
 \bfzero_{\calJ}& \bfI_{\bar{\calJ}}& \bfzero_{\calJ}& \bfzero_{\bar{\calJ}}\\
 \bfzero_{\calJ}& \bfzero_{\bar{\calJ}}& \bfI_{\calJ}& \bfzero_{\bar{\calJ}}\\
 \bfzero_{\calJ}& \bfzero_{\bar{\calJ}}& \bfzero_{\calJ}& \bfI_{\bar{\calJ}}
\end{pmatrix}
\begin{pmatrix}
\bfZ_{1\calJ}\\
\bfZ_{1\bar{\calJ}}\\
\bfZ_{2\calJ}\\
\bfZ_{2\bar{\calJ}}
\end{pmatrix}
\]
where $(\bfZ_1,\bfZ_2)=(\bfZ_{1\calJ},\bfZ_{1\bar{\calJ}},\bfZ_{2\calJ},\bfZ_{2\bar{\calJ}})\sim \text{CSN}_{2d,d(n-2)}(\bar{\bfSigma}_{{\calI}{\calI}},\bfvarrho_{\bcalI,\bcalI},\bar{\bfSigma}_{{\bcalI}{\bcalI};{\calI}})$ and by applying \eqref{eq: affine CSN}.
The second probability is computed in the same way.
\end{proof}

\bibliographystyle{chicago}
\bibliography{evt}
\end{document}